\documentclass[12pt]{article}

\usepackage{tikz}
\usetikzlibrary{automata}
\usetikzlibrary{decorations.pathreplacing}
\usepackage{forest}

\usepackage{color}

\usepackage{amsmath,amssymb,amsthm}
\usepackage{thmtools}

\usepackage[T1]{fontenc}

\usepackage[colorlinks=true,allcolors=.]{hyperref}
\usepackage[nameinlink]{cleveref}

\usepackage{authblk}

\usepackage{enumitem}
\usepackage{makecell}

\newtheorem{theorem}{Theorem}[section]
\newtheorem{lemma}[theorem]{Lemma}
\newtheorem{proposition}[theorem]{Proposition}
\newtheorem{corollary}[theorem]{Corollary}
\newtheorem{claim}[theorem]{Claim}

\newtheorem{conjecture}{Conjecture}
\newtheorem{question}{Question}

\theoremstyle{definition}

\newtheorem{example}[theorem]{Example}

\DeclareMathOperator{\ce}{\text{\normalfont ce}}
\DeclareMathOperator{\RT}{\text{\normalfont RT}}
\DeclareMathOperator{\Fact}{\text{\normalfont Fact}}

\DeclareMathOperator{\Occ}{\text{\normalfont Occ}}
\DeclareMathOperator{\Real}{\text{\normalfont Re}}

\renewcommand{\tt}[1]{\mathtt{#1}}

\newenvironment{subproof}[1][\proofname]{%
  \begin{proof}[#1]%
}{%
  \end{proof}%
}

\begin{document}

\title{The repetition threshold for ternary rich words}

\date{May 2025}

\author[1]{James D.~Currie}
\author[2]{Lucas Mol}
\author[3]{Jarkko Peltom\"aki}
\affil[1]{
Department of Mathematics and Statistics,
The University of Winnipeg,
Winnipeg, Canada,
\href{mailto:j.currie@uwinnipeg.ca}{j.currie@uwinnipeg.ca}.
}
\affil[2]{
Department of Mathematics and Statistics,
Thompson Rivers University,
Kamloops, Canada,
\href{mailto:lmol@tru.ca}{lmol@tru.ca}.
}
\affil[3]{
Information Technology,
\capitalring{A}bo Akademi University, 
Turku, Finland,
\href{mailto:r@turambar.org}{r@turambar.org}.
}

\maketitle

\begin{abstract}
In 2017, Vesti proposed the problem of determining the repetition threshold for infinite rich words, i.e., for infinite words in which all factors of length $n$ contain $n$ distinct nonempty palindromic factors. In 2020, Currie, Mol, and Rampersad proved a conjecture of Baranwal and Shallit that the repetition threshold for binary rich words is $2 + \sqrt{2}/2$. In this paper, we prove a structure theorem for $16/7$-power-free ternary rich words. Using the structure theorem, we deduce that the repetition threshold for ternary rich words is $1 + 1/(3 - \mu) \approx 2.25876324$, where $\mu$ is the unique real root of the polynomial $x^3 - 2x^2 - 1$.
\end{abstract}

\section{Introduction}

The study of repetitions in words goes back to the works of Thue at the beginning of the twentieth century~\cite{Thue1906,Thue1912}, which have been translated to English by Berstel~\cite{Berstel1995}.  Many extensions and variations of Thue's results have been proven since then; see, for example, the surveys of Ochem, Rao, and Rosenfeld~\cite{OchemRaoRosenfeld2018} and Rampersad and Shallit~\cite{RampersadShallit2016}.  We use standard notation and terminology related to repetitions in words in the remainder of this section; the unfamiliar reader should refer to \Cref{Section:Background}.

In this paper, we study repetitions in so-called rich words.  A \emph{palindrome} is a word that reads the same forwards and backwards. Droubay, Justin, and Pirillo~\cite{DroubayJustinPirillo2001} were first to observe that every word of length $n$ contains at most $n+1$ distinct palindromes as factors, including the empty word. A word of length $n$ is called \emph{rich} if it has $n+1$ distinct palindromes as factors; an infinite word is rich if all of its finite factors are rich.
Since their (implicit) introduction by Droubay, Justin, and Pirillo, rich words have been well-studied.  It is known that the language of infinite rich words contains several highly structured classes of words, including Sturmian words, episturmian words, and complementary symmetric Rote words (see~\cite{BlondinMasseEtAl2011,DroubayJustinPirillo2001}). Infinite rich words have been characterized in terms of a condition on complete returns to palindromes~\cite{GlenJustinWidmerZamboni2009}, in terms of a relation between factor and palindromic complexity~\cite{BucciDeLucaGlenZamboni2009}, and in terms of a condition on bispecial factors~\cite{BalkovaPelantovaStarosta2010}. For any fixed integer $k\geq 2$, the number of rich words of length $n$ over $k$ letters is known to grow superpolynomially~\cite{GuoShallitShur2016} and subexponentially~\cite{Rukavicka2017} in $n$.

Vesti~\cite{Vesti2017} proposed the problem of determining the repetition threshold for the language of infinite rich words over $k$ letters.  This problem has been resolved in the binary case through the combined effort of several authors.  Baranwal and Shallit~\cite{BaranwalShallit2019} constructed an infinite binary rich word with critical exponent $2+\sqrt{2}/2 \approx 2.707$ and conjectured that $2+\sqrt{2}/2$ is in fact the repetition threshold for infinite binary rich words.  This conjecture was confirmed by Currie, Mol, and Rampersad~\cite{CurrieMolRampersad2020}, who proved a structure theorem for $14/5$-power-free infinite binary rich words.  Roughly speaking, the structure theorem says that every $14/5$-power-free infinite binary rich word contains all of the factors of one of two specific infinite binary words (one being the word of Baranwal and Shallit).  Since both of these words turn out to be rich and have critical exponent $2+\sqrt{2}/2$, Baranwal and Shallit's conjecture follows from the structure theorem.

A general lower bound on the repetition threshold for the language of infinite rich words over $k$ letters follows from a result of Pelantov\'a and Starosta~\cite{PelantovaStarosta2013}, which says that every infinite rich word, over \emph{any} finite alphabet, contains infinitely many factors of exponent $2$.  It follows that the repetition threshold (and also the \emph{asymptotic} repetition threshold) for the language of infinite rich words over $k$ letters is at least $2$ for every $k\geq 2$.  In fact, Dvo\v{r}\'akov\'a, Klouda, and Pelantov\'a~\cite{DvorakovaKloudaPelantova2024} have recently shown that the asymptotic repetition threshold for the language of infinite rich  words over $k$ letters is equal to $2$ for all $k\geq 2$.  But to date, Vesti's problem of determining the (ordinary) repetition threshold for the language of infinite rich words over $k$ letters has only been resolved in the binary case.

In this paper, we resolve Vesti's problem in the ternary case by proving the following theorem.

\begin{theorem}\label{Theorem:RepetitionThreshold}
    The repetition threshold for the language of infinite ternary rich words is 
    \begin{equation*}
        1+\frac{1}{3 - \mu_1}\approx 2.25876324,
    \end{equation*}
    where $\mu_1$ is the unique real root of the polynomial $x^3-2x^2-1$.
\end{theorem}

In order to prove \Cref{Theorem:RepetitionThreshold}, we first prove a structure theorem for infinite $16/7$-power-free ternary rich words, similar to the one established by Currie, Mol, and Rampersad~\cite{CurrieMolRampersad2020} in the binary case.  Roughly speaking, we show that every infinite $16/7$-power-free ternary rich word ``looks like'' one specific word. 
Structure theorems like this have existed in the combinatorics on words literature from the very beginning. Thue~\cite{Thue1912} established structure theorems for infinite square-free ternary words avoiding various sets of factors.  Another well-known (and incredibly useful) structure theorem is the one for $7/3$-power-free binary words due to Karhum\"aki and Shallit~\cite{KarhumakiShallit2004}, which says that every infinite $7/3$-power-free binary word contains all factors of the Thue-Morse word.  In recent years, a surprising number of new structure theorems have been established; see~\cite{BadkobehOchem2023,BaranwalEtAl2023,Currie2023,CurrieOchemRampersadShallit2023,CurrieRampersad2023,DvorakovaOchemOpocenska2024}, for example.

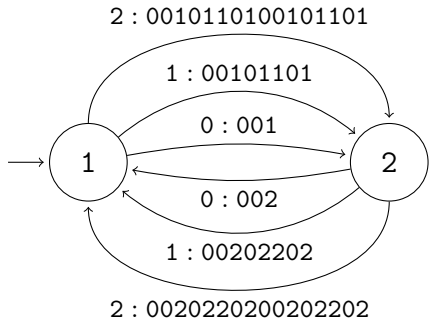
\begin{figure}
    \centering
    \begin{tikzpicture}[shorten >=2pt,initial text={}]
        \node[state,initial] (even) at (0,0) {$\tt{1}$};
        \node[state] (odd) at (4,0) {$\tt{2}$};
        \path[style=->]
        (even) edge[bend left=10] node[above]{\footnotesize $\tt{0}:\tt{001}$} (odd)
        (even) edge[bend left=40] node[above]{\footnotesize $\tt{1}:\tt{00101101}$} (odd)
        (even) edge[bend left=90] node[above]{\footnotesize $\tt{2}:\tt{0010110100101101}$} (odd)
        (odd) edge[bend left=10] node[below]{\footnotesize $\tt{0}:\tt{002}$} (even)
        (odd) edge[bend left=40] node[below]{\footnotesize $\tt{1}:\tt{00202202}$} (even)
        (odd) edge[bend left=90] node[below]{\footnotesize $\tt{2}:\tt{0020220200202202}$} (even)
        ;
    \end{tikzpicture}   
    \caption{The transducer $\tau$.  We will also refer to the related transducer $\overline{\tau}$, which has the same states and transitions, but starts in state $\tt{2}$.  }
    \label{Fig:tau}
\end{figure}

For our structure theorem, we let $f, g\colon \Sigma_3^*\rightarrow \Sigma_3^*$ be the morphisms defined by
\begin{align*}
    f(\tt{0})&=\tt{01} & g(\tt{0})&=\tt{20}\\
    f(\tt{1})&=\tt{022} & g(\tt{1})&=\tt{21}\\
    f(\tt{2})&=\tt{02} & g(\tt{2})&=\tt{2},
\end{align*}
and we let $\tau$ be the transducer drawn in \Cref{Fig:tau}.

\begin{theorem}[Structure Theorem]\label{Theorem:Structure}
    Suppose that $\mathbf{z}\in\Sigma_3^\omega$ is a $16/7$-power-free rich word.  Then for all $n\geq 0$, a suffix of $\mathbf{z}$ can be obtained from $\tau(g(f^n(\mathbf{x}_n)))$ by permuting the letters, where $\mathbf{x}_n\in\Sigma_3^\omega$.  In particular, $\mathbf{z}$ contains all of the factors obtained from $\tau(g(f^\omega(\tt{0})))$ by permuting the letters.
\end{theorem}

We then show that the word $\tau(g(f^\omega(\tt{0})))$ is rich and has critical exponent $1+1/(3-\mu_1)$. \Cref{Theorem:RepetitionThreshold} then follows easily from these two facts and \Cref{Theorem:Structure}.

The remainder of the paper is laid out as follows. In \Cref{Section:Background}, we describe the notation and terminology used in the paper, and provide a more comprehensive summary of results related to \Cref{Theorem:RepetitionThreshold}.  In \Cref{Section:Structure}, we prove \Cref{Theorem:Structure}.  In \Cref{Section:Richness}, we prove that $\tau(g(f^\omega(\tt{0})))$ is rich, and in \Cref{Section:CriticalExponent}, we prove that $\tau(g(f^\omega(\tt{0})))$ has critical exponent $1+1 / (3 - \mu_1)$.  Finally, in \Cref{Section:Conclusion}, we ask some questions related to Vesti's problem over larger alphabets.


\section{Notation, Terminology, and Background}
\label{Section:Background}

\subsection{Words}

An \emph{alphabet} is a nonempty finite set of symbols, which we refer to as \emph{letters}.  Throughout, we let $\Sigma_k$ denote the alphabet $\{\tt{0},\tt{1},\ldots,\tt{k-1}\}$.  A \emph{word} over an alphabet $A$ is a finite or infinite sequence of letters from $A$.   The \emph{length} of a finite word $w$, denoted by $|w|$, is the number of letters that make up $w$.  We let $\varepsilon$ denote the \emph{empty word}, which is the unique word of length $0$.  For a letter $a$, we let $|w|_a$ denote the number of occurrences of $a$ in $w$.

For a finite word $x$ and a finite or infinite word $y$, the \emph{concatenation} of $x$ and $y$, denoted by $xy$, is the word consisting of all of the letters of $x$ followed by all of the letters of $y$.  Suppose that a finite or infinite word $w$ can be written in the form $w=xyz$, where $x$, $y$, and $z$ are possibly empty words.  Then the word $y$ is called a \emph{factor} of $w$, the word $x$ is called a \emph{prefix} of $w$, and the word $z$ is called a \emph{suffix} of $w$.  If $x$ and $z$ are nonempty, then $y$ is called an \emph{internal factor} of $w$, and we say that $y$ \emph{appears internally} in $w$.  If $yz$ is nonempty, then $x$ is called a \emph{proper prefix} of $w$, and if $xy$ is nonempty, then $z$ is called a \emph{proper suffix} of $w$.  For an infinite word $\mathbf{u}$, the \emph{language} of $\mathbf{u}$, denoted by $\Fact(\mathbf{u})$, is the set of all finite factors of $\mathbf{u}$.

For a set of finite words $A$, we let $A^*$ denote the set of finite words obtained by concatenating elements of $A$, and we let $A^\omega$ denote the set of infinite words obtained by concatenating elements of $A$.  So in particular, $\Sigma_k^*$ is the set of finite words over $\Sigma_k$, and $\Sigma_k^\omega$ is the set of infinite words over $\Sigma_k$.

Let $\mathbf{w}$ be an infinite word, and let $w$ be a finite factor of $\mathbf{w}$.  We say that $w$ is \emph{recurrent} in $\mathbf{w}$ if it occurs infinitely many times in $\mathbf{w}$ and that $w$ is \emph{uniformly recurrent} in $\mathbf{w}$ if there is an integer $k$ such that every factor of $\mathbf{w}$ of length $k$ contains $w$.  We say that the infinite word $\mathbf{w}$ is \emph{(uniformly) recurrent} if every finite factor of $\mathbf{w}$ is (uniformly) recurrent.  A \emph{complete return word} to $w$ in $\mathbf{w}$ is a factor of $\mathbf{w}$ that contains $w$ as a proper prefix and a proper suffix but not as an internal factor.  If $w$ is recurrent in $\mathbf{w}$, then every occurrence of $w$ in $\mathbf{w}$ is followed by a successive occurrence, which gives rise to a complete return word to $w$.  Evidently, if $w$ is uniformly recurrent in $\mathbf{w}$, then there are only finitely many complete return words to $w$ in $\mathbf{w}$. We say that a factor of $w$ is \emph{unioccurrent} if it occurs exactly once in $w$ as a factor.

\subsection{Morphisms and Transducers}

For alphabets $A$ and $B$, a $\emph{morphism}$ from $A^{*}$ to $B^{*}$ is a function $h\colon A^{*} \rightarrow B^{*}$ that satisfies $h(uv) = h(u)h(v)$ for all words $u,v \in A^{*}$.  
Let $h\colon A^{*} \rightarrow A^{*}$ be a morphism. For all words $x\in A^*$, we define $h^{0}(x) = x$, and $h^n(x)=h(h^{n-1}(x))$ for all integers $n\geq 1$.  For a letter $a\in A$, we say that $h$ is \emph{prolongable} on $a$ if $h(a) = ax$ for some non-empty word $x \in X^{*}$ and $h^n(x)\neq \varepsilon$ for all $n\geq 0$. 
If $h$ is prolongable on $a$ with $h(a)=ax$, then it is easy to show that for every integer $n\geq 1$, we have
\begin{equation*}h^n(a) = axh(x)h^{2}(x)\cdots h^{n-1}(x).\end{equation*}
Thus, each $h^n(a)$ is a prefix of the infinite word
\begin{equation*}
h^\omega(a)=axh(x)h^2(x)h^3(x)\cdots.
\end{equation*}
Note that $h^\omega(a)$ is a fixed point of $h$, i.e., $h(h^\omega(a))=h^\omega(a)$, where the morphism $h$ is extended to infinite words in the natural way.

For the formal definition of finite-state transducer, see~\cite[Section~4.3]{AlloucheShallit2003}.  We only make use of the transducer $\tau$ drawn in \Cref{Fig:tau} (and the related transducer $\overline{\tau}$), whose behavior is very simple. These transducers simply alternate between the two states on every input letter.  In fact, one could define $\tau$ and $\overline{\tau}$ in terms of the morphisms $t_1,t_2\colon\Sigma_3^*\rightarrow \Sigma_3^*$ defined as follows:
\begin{align*}
    t_1(\tt{0})&=\tt{001} & t_2(\tt{0})&=\tt{002}\\
    t_1(\tt{1})&=\tt{00101101} & t_2(\tt{1})&=\tt{00202202}\\
    t_1(\tt{2})&=\tt{0010110100101101} & t_2(\tt{2})&=\tt{0020220200202202}.
\end{align*}
For a finite or infinite word $w=w_0w_1w_2w_3\cdots$ over $\Sigma_3$, where the $w_i$ are letters, we have
\begin{equation*}
\tau(w)=t_1(w_0)t_2(w_1)t_1(w_2)t_2(w_3)\cdots
\end{equation*}
and 
\begin{equation*}
\overline{\tau}(w)=t_2(w_0)t_1(w_1)t_2(w_2)t_1(w_3)\cdots.
\end{equation*}
Essentially, the output of $\tau$ is obtained from the input word by applying $t_1$ to all letters of even index and $t_2$ to all letters of odd index, and the output of $\overline{\tau}$ is obtained by applying $t_2$ to all letters of even index and $t_1$ to all letters of odd index.  We note that the well-known Arshon words can be defined in terms of similar maps (c.f.~\cite{Krieger2010}).

\subsection{Repetitions in Words}

Let $w=w_0w_1\cdots w_{n-1}$ be a finite word, where the $w_i$ are letters.  For an integer $p$ with $1\leq p\leq n$, we say that $p$ is a \emph{period} of $w$ if $w_{i+p}=w_i$ for all $0\leq i < n-p$.  In this case, we say that $n/p$ is an \emph{exponent} of $w$, and that $w$ is an \emph{$n/p$-power}.  The smallest period of $w$ is called \emph{the} period of $w$, and it corresponds to the largest exponent of $w$, which is called \emph{the} exponent of $w$.  

Now let $\mathbf{w}$ be a finite or infinite word, and let $\alpha>1$ be a real number.  We say that $\mathbf{w}$ is \emph{$\alpha$-power-free} if it contains no factors of exponent greater than or equal to $\alpha$, and that $\mathbf{w}$ is \emph{$\alpha^+$-power-free} if it contains no factors of exponent strictly greater than $\alpha$.  The \emph{critical exponent} of $\mathbf{w}$, denoted by $\ce(\mathbf{w})$, is defined by
\begin{equation*}
\ce(\mathbf{w})=\sup\{r\in \mathbb{Q} : \text{$\mathbf{w}$ has a factor of exponent $r$}\},
\end{equation*}
or equivalently by
\begin{equation*}
\ce(\mathbf{w})=\inf\{\alpha\in\mathbb{R} : \text{$\mathbf{w}$ is $\alpha$-power-free}\}.
\end{equation*}
Roughly speaking, the critical exponent of $\mathbf{w}$ describes the largest exponents among all factors of $\mathbf{w}$.


Given a language $L$ of infinite words, it is natural to try to find the infimum of the critical exponents among all words in $L$.  The words in $L$ with the smallest critical exponents are in some sense the ``least repetitive'' words in $L$.  The \emph{repetition threshold} of $L$, denoted by $\RT(L)$, is defined by
\begin{equation*}
\RT(L)=\inf\{\ce(\mathbf{w}) : \mathbf{w}\in L\}.
\end{equation*}

One of the most celebrated results in the area of repetitions in words is Dejean's theorem, which describes the repetition threshold of the language of all infinite words over $k$ letters:
\begin{equation*}
\RT(\Sigma_k^\omega)=\begin{cases}
    2, &\text{ if $k=2$;}\\
    7/4, &\text{ if $k=3$;}\\
    7/5, &\text{ if $k=4$;}\\
    k/(k-1), &\text{ if $k\geq 5$.}
\end{cases}
\end{equation*}
The case $k=2$ was proven by Thue~\cite{Thue1912}, and Dejean~\cite{Dejean1972} proved the case $k=3$ and correctly conjectured the remaining cases, which have been proven through the work of many different authors. 
In particular, Carpi~\cite{Carpi2007} proved all but finitely many cases, and the last cases were proven independently by Currie and Rampersad~\cite{CurrieRampersad2011} and Rao~\cite{Rao2011}.

Especially since the completion of the proof of Dejean's theorem, much work has been done on determining the repetition thresholds of various narrower languages.  The repetition threshold for the language of all Sturmian words was determined by Carpi and de Luca~\cite{CarpiDeLuca2000}, who showed that the Fibonacci word has the least critical exponent among all Sturmian words, namely $(5+\sqrt{5})/2$.  Sturmian words are episturmian, balanced, and rich, and research has been done on determining the repetition thresholds of $k$-ary words of these three types in the last decade.  For all $k\geq 2$, we let $E_k$, $B_k$, and $R_k$ denote the languages of $k$-ary episturmian, balanced, and rich words, respectively.  

For episturmian words, Dvo\v{r}\'akov\'a and Pelantov\'a~\cite{DvorakovaPelantova2024} recently established that 
\begin{equation*}
\RT(E_k)=2+\frac{1}{t_k-1}
\end{equation*}
for all $k\geq 2$, where $t_k$ is the unique positive root of the polynomial $x^k-x^{k-1}-\cdots-x-1$.  

For balanced words, the following is known:
\begin{equation*}
\RT(B_k)=\begin{cases}
    2+\frac{1+\sqrt{5}}{2}, &\text{ if $k=2$~\cite{CarpiDeLuca2000};}\\
    2+\frac{\sqrt{2}}{2}, &\text{ if $k=3$~\cite{RampersadShallitVandomme2019};}\\
    1+\frac{1+\sqrt{5}}{4}, &\text{ if $k=4$~\cite{RampersadShallitVandomme2019};}\\
    1+\frac{1}{k-3}, &\text{ if $5\leq k\leq 10$~\cite{BaranwalThesis,DolceDvorakovaPelantova2023};}\\
    1+\frac{1}{k-2}, &\text{ if $k=11$ or both $k\geq 12$ and $k$ is even~\cite{DvorakovaOpocenskaPelantovaShur2022}.}
\end{cases}
\end{equation*}
Dvo\v{r}\'akov\'{a}, Opo\v{c}ensk\'a, Pelantov\'{a}, and Shur~\cite{DvorakovaOpocenskaPelantovaShur2022} conjecture that $\RT(B_k)=1+1/(k-2)$ for all $k\geq 13$ with $k$ odd, but this conjecture remains open at the time of writing.

As described in the introduction, for rich words, it is known that $\RT(R_k)\geq 2$ for all $k\geq 2$, and that $\RT(R_2)=2+\sqrt{2}/2$.

It turns out that there are many infinite words $\mathbf{w}$ in which only ``short'' factors of $\mathbf{w}$ have exponent equal to (or even close to) the critical exponent of $\mathbf{w}$.  If we ignore such short factors, and consider only arbitrarily long factors of $\mathbf{w}$, then we obtain what is called the \emph{asymptotic critical exponent} of $\mathbf{w}$, first defined by Cassaigne~\cite{Cassaigne2008}.  It is denoted by $\ce^*(\mathbf{w})$ and defined by
\begin{equation*}
\ce^*(\mathbf{w})=\limsup_{n\rightarrow\infty}\{r\in \mathbb{Q} : \text{$\mathbf{w}$ has a factor of exponent $r$ and period $n$}\}.
\end{equation*}
The \emph{asymptotic repetition threshold} of $L$, denoted by $\RT^*(L)$, is defined by 
\begin{equation*}
\RT^*(L)=\inf\{\ce^*(\mathbf{w})\colon\ \mathbf{w}\in L\}.
\end{equation*}
For every language of infinite words $L$, we clearly have $\RT^*(L)\leq \RT(L)$.  We have equality for some languages $L$, and strict inequality for others.  For example, Cassaigne~\cite{Cassaigne2008} observed that $\RT^*(S)=\RT(S)$, where $S$ is the language of Sturmian words, but demonstrated that $\RT^*(\Sigma_k^\omega)=1<\RT(\Sigma_k^\omega)$ for all $k\geq 2$.  We refer the reader to the recent work of Dvo\v{r}\'{a}kov\'{a}, Klouda, and Pelantov\'{a}~\cite{DvorakovaKloudaPelantova2024} for a summary of results on the asymptotic repetition thresholds of episturmian, balanced, and rich words.  In particular, as mentioned in the introduction, they show that for rich words, we have $\RT^*(R_k)=2$ for all $k\geq 2$.

\section{The Structure Theorem}
\label{Section:Structure}

In this section, we prove \Cref{Theorem:Structure}.  Suppose that $\mathbf{z}$ is an infinite $16/7$-power-free ternary rich word.  
In \Cref{Subsection:FirstLayer}, we show that up to permutation of the letters, some suffix of $\mathbf{z}$ has the form $\tau(\mathbf{y})$ for some word $\mathbf{y}\in\Sigma_3^\omega$.  (In fact, we show this under the slightly weaker assumption that $\mathbf{z}$ is an infinite $7/3$-power-free ternary rich word.)  In \Cref{Subsection:SecondLayer}, we show that a suffix of $\mathbf{y}$ has the form $g(\mathbf{x})$ for some word $\mathbf{x}\in\Sigma_3^\omega$.  In \Cref{Subsection:ForbiddenFactors}, we describe several families of factors that cannot appear in $\mathbf{x}$ if the word $\tau(g(f^n(\mathbf{x})))$ is $16/7$-power-free and rich, where $n$ is any nonnegative integer.  Finally, in \Cref{Subsection:InnerLayers}, we show that if $\tau(g(f^n(\mathbf{x})))$ is $16/7$-power-free and rich for some $n\geq 0$ and $\mathbf{x}\in\Sigma_3^\omega$, then a suffix of $\mathbf{x}$ has the form $f(\mathbf{x}')$ for some word $\mathbf{x}'\in\Sigma_3^\omega$, which allows us to complete the proof of \Cref{Theorem:Structure} by mathematical induction.

\subsection{The First Layer}\label{Subsection:FirstLayer}

Throughout this subsection, let $\mathbf{z}$ be an infinite $7/3$-power-free ternary rich word. 
We first observe that $\mathbf{z}$ must contain at least one of the factors $\tt{001002}$, $\tt{112110}$, and $\tt{220221}$, as a computer backtracking search shows that a longest $7/3$-power-free ternary rich word containing none of these three factors has length $388$.\footnote{Code to verify all backtracking searches in this paper can be found at \url{https://github.com/japeltom/ternary-rich-words-verification}.}  
Note that $\tt{001002}$, $\tt{112110}$, and $\tt{220221}$ can be obtained from one another by permuting the letters.
Thus, by permuting the letters and removing a prefix of $\mathbf{z}$ if necessary, we assume henceforth that $\mathbf{z}$ has prefix $\tt{001002}$.

\begin{table}
\begin{center}
\begin{tabular}{c | c}
$p$ & \makecell{Length of a longest $7/3$-power-free \\ ternary rich word with prefix $p$}\\\hline
$\tt{102}$ & 152\\
$\tt{0011}$ & 498\\
$\tt{00100200}$ & 502
\end{tabular}
\end{center}
\caption{Some words that do not appear in $\mathbf{z}$.\label{Table:7over3}}
\end{table}

By performing further backtracking searches, we confirm that several short words, and all words obtained from them by permuting the letters, do not appear in $\mathbf{z}$.  The results are summarized in \Cref{Table:7over3}.  Our next lemma describes two more short words that do not appear in $\mathbf{z}$.  To prove it, we make use of the following well-known characterization of finite rich words.

\begin{lemma}[\cite{DroubayJustinPirillo2001}, c.f.\ \cite{GlenJustinWidmerZamboni2009}]\label{Lemma:UnioccurrentSuffix}
A finite word $w$ is rich if and only if every prefix (resp.~suffix) of $w$ has a unioccurrent palindromic suffix (resp.~prefix).
\end{lemma}

\begin{lemma}\label{Lemma:Neither12Nor21}
The word $\mathbf{z}$ contains neither $\tt{12}$ nor $\tt{21}$.
\end{lemma}

\begin{proof}
Suppose that $\mathbf{z}$ contains one of the factors $\tt{12}$ or $\tt{21}$. Let $p$ be the shortest prefix of $\mathbf{z}$ that ends in $\tt{12}$ or $\tt{21}$, and let $s$ be the suffix of $p$ of length $2$.  By the minimality of $p$, the reversal of $s$ does not appear in $p$, and it follows that the only nonempty palindromic suffix of $p$ has length $1$.  Since $p$ has prefix $\tt{001002}$, every word of length $1$ occurs more than once in $p$.  Hence $p$ does not have a unioccurrent palindromic suffix, and we conclude by \Cref{Lemma:UnioccurrentSuffix} that $p$ is not rich, which is a contradiction.
\end{proof}

Let 
\begin{align*}
A_\tt{1}&=\tt{001},\\
B_\tt{1}&=\tt{00101101},\\
A_\tt{2}&=\tt{002}, \text{ and}\\
B_\tt{2}&=\tt{00202202}.
\end{align*}
Note that $A_\tt{1}$ can be obtained from $A_\tt{2}$ (and vice versa) by swapping $\tt{1}$ and $\tt{2}$.  The same can be said for $B_\tt{1}$ and $B_\tt{2}$.  We will frequently need to deal with such pairs of words.  For a word $x\in\Sigma_3^*$, the \emph{sister} of $x$ is the word obtained from $x$ by swapping the letters $\tt{1}$ and $\tt{2}$.  With this terminology, the words $A_\tt{1}$ and $A_\tt{2}$ are sisters, as are the words $B_\tt{1}$ and $B_\tt{2}$.  Note also that for every word $w\in\Sigma_3^*$, the words $\tau(w)$ and $\overline{\tau}(w)$ are sisters.

Note that $A_\tt{1}$, $B_\tt{1}$, $A_\tt{2}$, and $B_\tt{2}$ have common prefix $\tt{00}$, and that $\tt{00}$ occurs only as a prefix of $A_\tt{1}$, $B_\tt{1}$, $A_\tt{2}$, or $B_\tt{2}$ in a word belonging to $\{A_\tt{1},B_\tt{1},A_\tt{2},B_\tt{2}\}^*$.

\begin{lemma}\label{Lemma:FirstLayer1}
$\mathbf{z}\in \{A_\tt{1},B_\tt{1},A_\tt{2},B_\tt{2}\}^\omega$.
\end{lemma}

\begin{proof}
By assumption, $\mathbf{z}$ begins with $\tt{00}$.  In fact, a computer backtracking search shows that the longest $7/3$-power-free ternary rich word with no $\tt{00}$ has length $57$, so that the factor $\tt{00}$ is recurrent in $\mathbf{z}$.  Thus, it suffices to show that the only complete returns to $\tt{00}$ in $\mathbf{z}$ are $A_\tt{1}\tt{00}$, $B_\tt{1}\tt{00}$, $A_\tt{2}\tt{00}$, or $B_\tt{2}\tt{00}$.  Since the cube $\tt{000}$ does not appear in $\mathbf{z}$, every occurrence of $\tt{00}$ in $\mathbf{z}$ must be followed by $\tt{1}$ or $\tt{2}$.  We show that the only complete returns to $\tt{00}$ starting with $\tt{001}$ in $\mathbf{z}$ are $A_\tt{1}\tt{00}$ and $B_\tt{1}\tt{00}$.  The argument that the only complete returns to $\tt{00}$ starting with $\tt{002}$ in $\mathbf{z}$ are $A_\tt{2}\tt{00}$ and $B_\tt{2}\tt{00}$ is symmetric.

Consider the tree drawn in \Cref{Fig:Tree1}, which shows all possible complete returns to $\tt{00}$ in $\mathbf{z}$ starting with $\tt{001}$.  From \Cref{Lemma:Neither12Nor21}, we know that $\tt{1}$ is never followed by $\tt{2}$ in $\mathbf{z}$, so these branches are suppressed in the tree.  We explain below why the words corresponding to the red leaves in the tree cannot appear in $\mathbf{z}$.
\begin{itemize}
\item $\tt{001010}$ has the $5/2$-power $\tt{01010}$ as a suffix.
\item $\tt{00101100}$ has a permutation of the word 
$\tt{0011}$
from \Cref{Table:7over3} as a suffix.
\item $\tt{0010110101}$ has the $5/2$-power $\tt{10101}$ as a suffix.
\item $\tt{0010110102}$ has the word 
$\tt{102}$ from \Cref{Table:7over3} as a suffix.
\item $\tt{001011011}$ has the $7/3$-power $\tt{1011011}$ as a suffix.
\item $\tt{00101102}$ has 
the word $\tt{102}$ from \Cref{Table:7over3} as a suffix.
\item $\tt{0010111}$ has the cube $\tt{111}$ as a suffix.
\item $\tt{00102}$ has the word 
$\tt{102}$ from \Cref{Table:7over3} as a suffix.
\item $\tt{0011}$ is in \Cref{Table:7over3}.
\end{itemize}
\begin{figure}
\centering
\begin{forest}
[
$\tt{001}$,for tree={grow=0,l=1.8cm}
	[
		$\tt{1}$,draw,fill=red
	]
	[
		$\tt{0}$
		[
			$\tt{2}$,draw,fill=red
		]
		[
			$\tt{1}$
			[
				$\tt{1}$
				[
					$\tt{1}$,draw,fill=red
				]
				[
					$\tt{0}$
					[
						$\tt{2}$,draw,fill=red
					]
					[
						$\tt{1}$
						[
							$\tt{1}$,draw,fill=red
						]
						[
							$\tt{0}$
							[
								$\tt{2}$,draw,fill=red
							]
							[
								$\tt{1}$,draw,fill=red
							]
							[
								$\tt{0}$,draw,fill=green
							]
						]
					]
					[
						$\tt{0}$,draw,fill=red
					]
				]
			]
			[
				$\tt{0}$,draw,fill=red
			]
		]
		[
       		$\tt{0}$,draw,fill=green
       ]
	]
]
\end{forest}
    \caption{The tree showing all possible complete returns to $\tt{00}$ in $\mathbf{z}$ starting with $\tt{001}$.}
    \label{Fig:Tree1}
\end{figure}
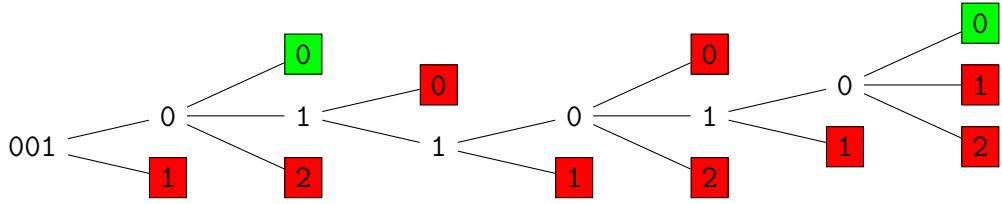
This means that the words corresponding to the green leaves in the tree are the only possible complete returns to $\tt{00}$ starting with $\tt{001}$ in $\mathbf{z}$.  Therefore, we conclude that $\mathbf{z}\in\{A_\tt{1},B_\tt{1},A_\tt{2},B_\tt{2}\}^\omega$, as desired.
\end{proof}

Let $q\colon \{\tt{a_1},\tt{b_1},\tt{a_2},\tt{b_2}\}^*\rightarrow \{\tt{0},\tt{1},\tt{2}\}^*$ be the morphism that sends each lowercase letter to the word denoted by the corresponding uppercase letter.  Then \Cref{Lemma:FirstLayer1} says that $\mathbf{z}=q(\mathbf{v})$ for some word $\mathbf{v}\in\{\tt{a_1},\tt{b_1},\tt{a_2},\tt{b_2}\}^\omega$.

\begin{lemma}\label{Lemma:SecondLayer}
    The word $\mathbf{v}$ corresponds to an infinite walk on the directed graph $G$ drawn in \Cref{Fig:vgraph}, and contains neither $\tt{b_1}\tt{b_1}\tt{b_1}$ nor $\tt{b_2}\tt{b_2}\tt{b_2}$.
\end{lemma}

\begin{proof}
    First of all, if $\mathbf{v}$ contains $\tt{b_1}\tt{b_1}\tt{b_1}$ or $\tt{b_2}\tt{b_2}\tt{b_2}$, then $\mathbf{z}$ contains the cube $B_{\tt{1}}B_{\tt{1}}B_{\tt{1}}$ or its sister, contradicting the assumption that $\mathbf{z}$ is $7/3$-power-free.

    We know from \Cref{Lemma:FirstLayer1} that $\mathbf{v}\in\{\tt{a_1},\tt{b_1},\tt{a_2},\tt{b_2}\}^\omega$.  So to show that $\mathbf{v}$ corresponds to an infinite walk on $G$, it suffices to show that the following factors do not appear in $\mathbf{v}$:
    \begin{itemize}
        \item $\tt{a_1}\tt{a_1}$, $\tt{a_1}\tt{b_1}$, $\tt{a_1}\tt{a_2}$, and $\tt{b_1}\tt{a_1}$; and
        \item $\tt{a_2}\tt{a_2}$, $\tt{a_2}\tt{b_2}$, $\tt{a_2}\tt{a_1}$, and $\tt{b_2}\tt{a_2}$.
    \end{itemize}
    By symmetry, we need only show that the first four factors do not appear in $\mathbf{v}$.
    \begin{itemize}
        \item If $\mathbf{v}$ contains $\tt{a_1}\tt{a_1}$, then $\mathbf{z}$ contains the $8/3$-power $A_\tt{1}A_\tt{1}\tt{00}=\tt{00100100}$.
        \item If $\mathbf{v}$ contains $\tt{a_1}\tt{b_1}$, then $\mathbf{z}$ contains $A_\tt{1}B_\tt{1}$, which has the $7/3$-power $\tt{0010010}$ as a prefix.
        \item If $\mathbf{v}$ contains $\tt{b_1}\tt{a_1}$, then $\mathbf{z}$ contains $B_\tt{1}A_\tt{1}\tt{00}$, which has the $7/3$-power $\tt{0100100}$ as a suffix.
        \item If $\mathbf{v}$ contains $\tt{a_1}\tt{a_2}$, then $\mathbf{z}$ contains $A_\tt{1}A_\tt{2}\tt{00}=\tt{00100200}$, which is in Table 1.
    \end{itemize}
    This completes the proof.
\end{proof}

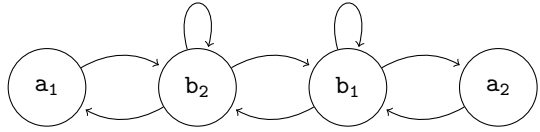
\begin{figure}
    \centering
    \begin{tikzpicture}[shorten >=2pt]
        \node[state] (A1) at (0,0) {\footnotesize $\tt{a_1}$};
        \node[state] (B2) at (2,0) {\footnotesize $\tt{b_2}$};
        \node[state] (B1) at (4,0) {\footnotesize $\tt{b_1}$};
        \node[state] (A2) at (6,0) {\footnotesize $\tt{a_2}$};
        \path[style=->]
        (A1) edge[bend left=30] (B2)
        (B2) edge[bend left=30] (A1)
        (B2) edge[bend left=30] (B1)
        (B1) edge[bend left=30] (B2)
        (B1) edge[bend left=30] (A2)
        (A2) edge[bend left=30] (B1)
        (B2) edge[loop above] (B2)
        (B1) edge[loop above] (B1)
        ;
    \end{tikzpicture}   
    \caption{The directed graph $G$ showing possible transitions between letters in $\mathbf{v}$.}
    \label{Fig:vgraph}
\end{figure}

Let $\mathbf{v}'$ be the word obtained from $\mathbf{v}$ by replacing each occurrence of the factor $\tt{b_1}\tt{b_1}$ with the new letter $\tt{c_1}$, and each occurrence of the factor $\tt{b_2}\tt{b_2}$ with the new letter $\tt{c_2}$.  In other words, we have $\mathbf{v}=p(\mathbf{v}')$, where $p\colon \{\tt{a_1},\tt{b_1},\tt{c_1},\tt{a_2},\tt{b_2},\tt{c_2}\}^*\rightarrow \{\tt{a_1},\tt{b_1},\tt{a_2},\tt{b_2}\}^*$ is the morphism that sends $\tt{c_1}$ to $\tt{b_1}\tt{b_1}$ and $\tt{c_2}$ to $\tt{b_2}\tt{b_2}$, and all other letters to themselves.
From \Cref{Lemma:SecondLayer}, we see that $\mathbf{v}'$ corresponds to an infinite walk on the 
graph $G'$ drawn in \Cref{Fig:v'graph}.  Note in particular that $\mathbf{v}'$ alternates between letters of index $\tt{1}$ and index $\tt{2}$.  It follows that we can write $\mathbf{v}'=\sigma(\mathbf{v}'')$, where $\textbf{v}''$ is an infinite word over $\{\tt{a,b,c}\}$ and $\sigma\colon \{\tt{a,b,c}\}\rightarrow \{\tt{a_1},\tt{b_1},\tt{c_1},\tt{a_2},\tt{b_2},\tt{c_2}\}$ is the map that adds index $\tt{1}$ and index $\tt{2}$ alternately to each letter of the input word, starting with $\tt{1}$.


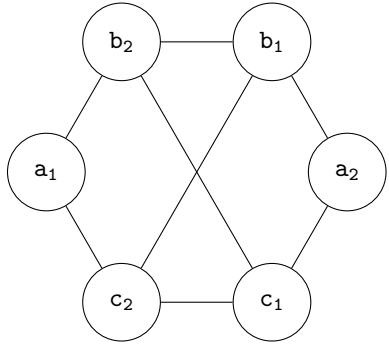
\begin{figure}
    \centering
    \begin{tikzpicture}
        \node[state] (A1) at (180:2) {\footnotesize $\tt{a_1}$};
        \node[state] (B2) at (120:2) {\footnotesize $\tt{b_2}$};
        \node[state] (B1) at (60:2) {\footnotesize $\tt{b_1}$};
        \node[state] (A2) at (0:2) {\footnotesize $\tt{a_2}$};
        \node[state] (C2) at (240:2) {\footnotesize $\tt{c_2}$};
        \node[state] (C1) at (300:2) {\footnotesize $\tt{c_1}$};
        \path[style=-]
        (A1) edge (B2)
        (B2) edge (B1)
        (B1) edge (A2)
        (A1) edge (C2)
        (C2) edge (C1)
        (C1) edge (A2)
        (C2) edge (B1)
        (C1) edge (B2)
        ;
    \end{tikzpicture}   
    \caption{The graph $G'$ showing possible transitions between letters in $\mathbf{v}'$.}
    \label{Fig:v'graph}
\end{figure}

We have shown that
\begin{equation*}
    \mathbf{z}=q(\mathbf{v})=q(p(\mathbf{v}'))=q(p(\sigma(\mathbf{v}''))),
\end{equation*}
where $\mathbf{v}''$ is an infinite word over $\{\tt{a,b,c}\}$.  In the sequel, we prefer to work with a word over the alphabet $\Sigma_3=\{\tt{0,1,2}\}$ instead of $\{\tt{a,b,c}\}$, so we write $\mathbf{v}''=r(\mathbf{y})$, where $\mathbf{y}\in\{\tt{0,1,2}\}^\omega$ and $r\colon \Sigma_3^*\rightarrow \{\tt{a,b,c}\}$ is defined by $r(\tt{0})=\tt{a}$, $r(\tt{1})=\tt{b}$, and $r(\tt{2})=\tt{c}$.  Thus we have
\begin{equation*}
    \mathbf{z} = q(p(\sigma(r(\mathbf{y})))).
\end{equation*}
We can express the composition of the maps $r$, $\sigma$, $p$, and $q$ as the single finite-state transducer $\tau$, which gives the following result.

\begin{proposition}\label{Prop:FirstLayer}
    We can write $\mathbf{z}=\tau(\mathbf{y})$, where $\mathbf{y}\in\Sigma_3^\omega$, and $\tau$ is the transducer drawn in \Cref{Fig:tau}.
\end{proposition}

\subsection{The Second Layer}\label{Subsection:SecondLayer}

Throughout this subsection, let $\textbf{z}$ be an infinite $16/7$-power-free ternary rich word.  Since $16/7<7/3$, we may use all of the results from \Cref{Subsection:FirstLayer}.  In particular, we may assume without loss of generality that $\mathbf{z}$ starts with $\tt{001002}$, and we know from \Cref{Prop:FirstLayer} that $\mathbf{z}=\tau(\mathbf{y})$ for some word $\mathbf{y}\in\Sigma_3^\omega$, where $\tau$ is the transducer drawn in \Cref{Fig:tau}.  In this subsection, we begin to describe the structure of the word $\textbf{y}$.

First of all, by backtracking, we confirm that several short words do not appear in $\mathbf{y}$.  The results are summarized in \Cref{Table:16over7}.

\begin{table}
\begin{center}
\begin{tabular}{c | c}
$p$ & \makecell{Length of a longest word $y\in \Sigma_3^*$ with prefix $p$ \\
such that $\tau(y)$ is a $16/7$-power-free rich word} \\\hline
$\tt{201}$ & 141\\
$\tt{210}$ & 144\\
$\tt{211}$ & 101
\end{tabular}
\end{center}
\caption{Some words that do not appear in $\mathbf{y}$. \label{Table:16over7}}
\end{table}

\begin{proposition}\label{Prop:SecondLayer}
    A suffix of $\mathbf{y}$ has the form $g(\mathbf{x})$ for some word $\mathbf{x}\in\Sigma_3^\omega$.
\end{proposition}

\begin{proof}
    A computer backtracking search shows that the longest word $y\in\{\tt{0,1}\}^*$ such that $\tau(y)$ is $16/7$-power-free and rich has length $18$.  So the letter $\tt{2}$ appears infinitely many times in $\mathbf{y}$.  Thus, it suffices to show that the only complete return words to $\tt{2}$ in $\mathbf{y}$ are $g(\tt{0})\tt{2}=\tt{202}$, $g(\tt{1})\tt{2}=\tt{212}$, and $g(\tt{2})\tt{2}=\tt{22}$.

    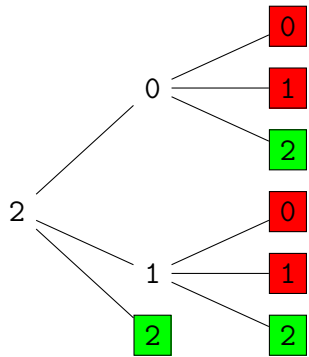
\begin{figure}
        \centering
        \begin{forest}
            [
            $\tt{2}$,for tree={grow=0,l=1.8cm}
                [
               	$\tt{2}$,draw,fill=green
                ]
            	[
            		$\tt{1}$
                    [
                    $\tt{2}$, draw, fill=green
                    ]
                    [
                    $\tt{1}$,draw,fill=red
                    ]
                    [
                    $\tt{0}$,draw,fill=red
                    ]
            	]
            	[
            		$\tt{0}$
                    [
                    $\tt{2}$, draw, fill=green
                    ]
                    [
                    $\tt{1}$,draw,fill=red
                    ]
                    [
                    $\tt{0}$,draw,fill=red
                    ]
            	]
            ]
        \end{forest}
        \caption{The tree showing all possible complete returns to $\tt{2}$ in $\mathbf{y}$.}
        \label{Fig:gTree}
    \end{figure}

    Consider the tree drawn in \Cref{Fig:gTree}, which shows all possible complete returns to $\tt{2}$ in $\mathbf{y}$.  We explain why the words corresponding to the red leaves in the tree cannot appear in $\mathbf{y}$.
    \begin{itemize}
        \item $\tt{200}$ has suffix $\tt{00}$.  If $\mathbf{y}$ contains $\tt{00}$, then $\textbf{z}$ contains $\tau(\tt{00})\tt{00}=\tt{00100200}$ or its sister.  Since $\tt{00100200}$ is in \Cref{Table:7over3}, this is impossible.
        \item $\tt{201}$ is in \Cref{Table:16over7}.
        \item $\tt{210}$ is in \Cref{Table:16over7}.
        \item $\tt{211}$ is in \Cref{Table:16over7}.
    \end{itemize}
This means that the words corresponding to the green leaves in the tree are the only possible complete return words to $\tt{2}$ in $\mathbf{y}$.  Therefore, we conclude that a suffix of $\mathbf{y}$ is in $\{\tt{20},\tt{21},\tt{2}\}^\omega$, as desired.
\end{proof}

\subsection{Forbidden Factors in the Inner Layers}\label{Subsection:ForbiddenFactors}

In this subsection, we describe several families of factors that cannot appear in $\mathbf{x}$ if the word $\tau(g(f^n(\mathbf{x})))$ is $16/7$-power-free and rich, where $n$ is any nonnegative integer. \Cref{Lemma:BadCubes} and \Cref{Lemma:ForbiddenRepetitions} describe factors in $\mathbf{x}$ that lead to repetitions in $\tau(g(f^n(\mathbf{x})))$ with exponent at least $16/7$, while \Cref{Lemma:ForbiddenRichness} describes factors in $\mathbf{x}$ that lead to non-richness in $\tau(g(f^n(\mathbf{x})))$.

Note that when the transducer $\tau$ is applied to an $r$-power $x$, the alternation of $\tau$ between $\tt{1}$'s and $\tt{2}$'s may break the repetition in $x$.  For example, if $x=\tt{000}$, then $x$ is a cube, but $\tau(x)=\tt{001002001}$ is only a $3/2$-power.  Note, however, that if the period of $x$ is even, then the alternation between $\tt{1}$'s and $\tt{2}$'s will line up after each period.  For example, $x=\tt{0101}$ is a square with period $2$, and we see that $\tau(x)=\tt{0010020220200100202202}$ is also a square.

If $x\in\Sigma_3^*$ is an $r$-power of the form $x=p^r$, where $p$ has an even number of $\tt{2}$'s, then we say that $x$ is an \emph{even-$r$-power}.  We say that a word is \emph{even-$r$-power-free} if it contains no even-$r$-powers.

\begin{lemma}\label{Lemma:BadPeriods}
    Let $x\in\Sigma_3^*$ be an even-$r$-power, and write $x=p^r$, where $p$ has an even number of $\tt{2}$'s.  Then for all $n\geq 0$, the word $\tau(g(f^n(x)))$ has period $|\tau(g(f^n(p)))|$.
\end{lemma}

\begin{proof}
    We first claim that $f^n(p)$ has an even number of $\tt{2}$'s for all $n\geq 0$.  Note that $f(\tt{0})$ and $f(\tt{1})$ both have an even number of $\tt{2}$'s, so the proof of the claim can be completed by a straightforward induction on $n$.

    Now let $n\geq 0$.  Since $f^n(p)$ has an even number of $\tt{2}$'s, and $g(\tt{0})$ and $g(\tt{1})$ both have even length, it follows that $g(f^n(p))$ has even length.  So by the observation preceding the lemma, we conclude that $\tau(g(f^n(x)))$ has period $|\tau(g(f^n(p)))|$.
\end{proof}

\begin{lemma}\label{Lemma:BadCubes}
    Let $\mathbf{x}\in\Sigma_3^\omega$, and suppose for some $n\geq 0$ that $\tau(g(f^n(\mathbf{x})))$ is $16/7$-power-free.  Then $\mathbf{x}$ is $5$-power-free and even-$3$-power-free.
\end{lemma}

\begin{proof}
    First, suppose towards a contradiction that $\mathbf{x}$ contains a $5$-power, say $x=p^5$.  Then we can write $x=(p^2)^{5/2}$, and $p^2$ has an even number of $\tt{2}$'s.  So by \Cref{Lemma:BadPeriods}, the word $\tau(g(f^n(x)))$ has period $|\tau(g(f^n(p^2)))|$ and exponent $5/2$. This contradicts the assumption that $\tau(g(f^n(\mathbf{x})))$ is $16/7$-power-free, so we see that $\mathbf{x}$ is $5$-power-free.

    Similarly, suppose that $\mathbf{x}$ contains an even-$3$-power, say $x=p^3$, where $p$ has an even number of $\tt{2}$'s.  Then by \Cref{Lemma:BadPeriods}, the word $\tau(g(f^n(x)))$ has period $|\tau(g(f^n(p)))|$ and exponent $3$.  This is a contradiction, and we conclude that $\mathbf{x}$ is even-$3$-power-free.
\end{proof}

For finite words $u$ and $v$ over an alphabet $\Sigma$, we write $u\preceq v$ if $|u|_a\leq |v|_a$ for all $a\in\Sigma$.  Note that non-erasing morphisms preserve the relation $\preceq$, and that $u\preceq v$ implies $|u|\leq |v|$.  Finally, note that if $u,v\in\Sigma_3^*$ and $u\preceq v$, then $|\tau(u)|\leq |\tau(v)|$.

\begin{lemma}\label{Lemma:LengthComp}
    For all $n\geq 0$, we have
    \begin{equation*}
        |f^n(\tt{2})|\leq |f^n(\tt{0})|\leq |f^n(\tt{1})|\leq 2|f^n(\tt{2})| 
    \end{equation*}
    and
    \begin{equation*}
        |\tau(g(f^n(\tt{2})))|\leq |\tau(g(f^n(\tt{0})))|\leq |\tau(g(f^n(\tt{1})))|\leq 2|\tau(g(f^n(\tt{2})))|.
    \end{equation*}
\end{lemma}

\begin{proof}
    We show only that $|f^n(\tt{2})|\leq |f^n(\tt{0})|$ and $|\tau(g(f^n(\tt{2})))|\leq |\tau(g(f^n(\tt{0})))|$ for all $n\geq 0$.  (The other inequalities can be proven in a similar manner.)  Both inequalities are verified by inspection for $n\in\{0,1\}$.  When $n=2$, we have
    \begin{equation*}
        f^2(\tt{2})=\tt{0102}\preceq \tt{01022}=f^2(\tt{0}),
    \end{equation*}
    so that $f^n(\tt{2})\preceq f^n(\tt{0})$ for all $n\geq 2$ by induction, and the desired inequalities follow immediately from the observations preceding the lemma.
\end{proof}

\begin{lemma}\label{Lemma:ForbiddenRepetitions}
    Let $\mathbf{x}\in\Sigma_3^\omega$, and suppose for some $n\geq 0$ that $\tau(g(f^n(\mathbf{x})))$ is $16/7$-power-free and rich.  Then no factor from the set $F=F_1\cup F_2\cup F_3\cup F_4$ appears internally in $\mathbf{x}$, where
    \begin{align*}
        F_1&=\{ \tt{00},\tt{11},\tt{212},\tt{0101},\tt{1010},\tt{2222},\tt{1222}, \tt{2221},\tt{022022},\tt{220220}\},\\
        F_2&=\{\tt{202202},\tt{1022021},\tt{1202201}, \tt{1}(\tt{20102})^2\tt{1}, (\tt{021012})^{13/6},(\tt{012021})^{13/6},\\
        & \hspace{1cm}   (\tt{21012010})^{21/8}, (\tt{21012210120})^{27/11}, (\tt{2101221012010})^{31/13}\},\\
        F_3&=\{\tt{2}(\tt{2101})^{17/4}\tt{2},(\tt{2101210122101})^{31/13}\}, \text{ and}\\
        F_4&=\{(\tt{0222})^{17/4}\tt{1},(\tt{22010})^{12/5},(\tt{022201})^{29/6},(\tt{0222010222})^{12/5},\\
        & \hspace{1cm}  (\tt{0222022201})^{12/5},(\tt{0222010222010222022201})^{5/2}\}.
    \end{align*}
\end{lemma}

\begin{proof}  
    Let $w\in F$, and suppose that $w$ appears internally in $\mathbf{x}$.  Then the word $\tau(g(f^n(\mathbf{x})))$ contains the word $\tau(g(f^n(awb)))$ or its sister for some $a,b\in\{\tt{0,1,2}\}$.  The general idea is to show that for all $a,b\in\{\tt{0,1,2}\}$, the word $\tau(g(f^n(awb)))$ has a factor of exponent at least $16/7$, which is a contradiction.  We do this for each word in $F_1$ below.  The proofs for the words in $F_2$, $F_3$, and $F_4$ are similar, and are omitted.\footnote{Please consult \url{https://github.com/japeltom/ternary-rich-words-verification} for the arguments in the omitted cases.}  Sometimes, we show the stronger statement that for all $b\in\{\tt{0,1,2}\}$, the word $\tau(g(f^n(wb))$ has a factor of exponent at least $16/7$, which implies that $w$ does not appear in $\mathbf{x}$ at all (i.e., neither internally nor as a prefix).
    
    $\tt{00}$: For $n=0$, we check directly that for all $b\in\{\tt{0,1,2}\}$, the word $\tau(g(\tt{00}b))$ contains a factor of exponent at least $16/7$.  So we may assume that $n\geq 1$.  For all $b\in\{\tt{0,1,2}\}$, we check that the word $f(\tt{00}b)$ contains the even-$5/2$-power $\tt{01010}$.  But this means that $\tau(g(f^n(\tt{00}b)))$ contains the word $\tau(g(f^{n-1}(\tt{01010})))$ or its sister.  By \Cref{Lemma:BadPeriods}, the word $\tau(g(f^{n-1}(\tt{01010})))$ has period $|\tau(g(f^{n-1}(\tt{01})))|$.  Further, by \Cref{Lemma:LengthComp}, we have $|\tau(g(f^{n-1}(\tt{01})))|\leq 3|\tau(g(f^{n-1}(\tt{0})))|$, so that $\tau(g(f^{n-1}(\tt{01010})))$ has exponent at least $7/3$.

    $\tt{11}$: For $n=0$, we check directly that for all $b\in\{\tt{0,1,2}\}$, the word $\tau(g(\tt{11}b))$ contains a factor of exponent at least $16/7$.  So we may assume that $n\geq 1$.  For all $b\in\{\tt{0,1,2}\}$, we check that the word $f(\tt{11}b)$ contains the even-$7/3$-power $\tt{0220220}$.  But this means that $\tau(g(f^n(\tt{11}b)))$ contains the word $\tau(g(f^{n-1}(\tt{0220220})))$ or its sister, and by \Cref{Lemma:BadPeriods} and \Cref{Lemma:LengthComp}, the word $\tau(g(f^{n-1}(\tt{0220220})))$ has exponent at least $7/3$.  

    $\tt{212}$:  For $n=0,1,2$, we check directly that for all $b\in\{\tt{0,1,2}\}$, the word $\tau(g(f^n(\tt{212}b)))$ has a factor of exponent at least $16/7$.  So we may assume that $n\geq 3$.  For all $b\in\{\tt{0,1,2}\}$, we check that the word $f^3(\tt{212}b)$ contains the even-$31/13$-power $p^{31/13}$, where $p=\tt{2010201022010}$.  This means that $\tau(g(f^n(\tt{212}b)))$ contains $\tau(g(f^{n-3}(p^{31/13})))$ or its sister.  We write $p^{31/13}=(p_1p_2p_3)^2p_1$,
    where $p_1=\tt{20102}$, $p_2=\tt{0102}$, and $p_3=\tt{2010}$.    Observe that $p_2,p_3\preceq p_1$, so that $|\tau(g(f^{n-3}(p_1p_2p_3)))| \leq 3|\tau(g(f^{n-3}(p_1)))|.$  From \Cref{Lemma:BadPeriods}, it follows that $\tau(g(f^{n-3}(p^{31/13})))$ has exponent at least $7/3$.

    $\tt{2222}$: For $n=0,1$, we check that for all $b\in\{\tt{0,1,2}\}$, the word $\tau(g(f^n(\tt{2222}b)))$ has a factor of exponent at least $16/7$.  So we may assume that $n\geq 2$.  For all $b\in\{\tt{0,1,2}\}$, we check that $f^2(\tt{2222}b)$ contains the $5$-power $(\tt{0102})^5$.  Thus $\tau(g(f^n(\tt{2222}a)))$ contains $\tau(g(f^{n-2}((\tt{0102})^5)))$ or its sister, and by \Cref{Lemma:BadPeriods}, the word $\tau(g(f^{n-2}((\tt{0102})^5)))$ has exponent $5/2$.

    $\tt{1222}$: We have already seen that $\mathbf{x}$ does not contain the factor $\tt{2222}$, so it suffices to show that for all $b\in\{\tt{0,1}\}$, the word $\tau(g(f^n(\tt{1222}b)))$ contains a factor of exponent at least $16/7$.  We check this directly for $n\leq 3$, so we may assume that $n\geq 4$.  For all $b\in\{\tt{0,1}\}$, we check that $f^4(\tt{1222}b)$ contains the even-$19/8$-power $p^{19/8}$, where $p=(\tt{02010220102010220102})^2$.  We write $p^{19/8}=(p_1p_2p_3)^2p_1$,
    where $p_1=\tt{020102201020102}$, $p_2=\tt{201020201022}$, and $p_3=\tt{0102010220102}$.  Observe that $p_2,p_3\preceq p_1$.  So by \Cref{Lemma:BadPeriods}, we see that $\tau(g(f^{n-4}(p^{19/8})))$ has exponent at least $7/3$.

    $\tt{2221}$: We have already seen that $\mathbf{x}$ does not contain the factor $\tt{2222}$, so it suffices to show that for all $a\in\{\tt{0,1}\}$ and $b\in\{\tt{0,1,2}\}$, the word $\tau(g(f^n(a\tt{2221}b)))$ has a factor of exponent at least $16/7$.  We check this directly for $n\leq 3$, so we may assume that $n\geq 4$.  For all $a\in\{\tt{0,1}\}$ and $b\in\{\tt{0,1,2}\}$, we check that $f^4(a\tt{2221}b)$ contains the even-$19/8$-power $p^{19/8}$, where $p=(\tt{20102010220102020102})^2$.  But one can show that $\tau(g(f^{n-4}(p^{19/8})))$ has exponent at least $7/3$ as in the proof for $\tt{1222}$.

    $\tt{1010}$: We have already seen that $\mathbf{x}$ contains no $\tt{00}$, so it suffices to show that for all $b\in\{\tt{1,2}\}$, the word $\tau(g(f^n(\tt{1010}b)))$ has a factor of exponent at least $16/7$.  We check this directly for $n=0$, so we may assume that $n\geq 1$.  For all $b\in\{\tt{1,2}\}$, we check that the word $f(\tt{1010}b)$ contains the even-$12/5$-power $(\tt{02201})^{12/5}$.  But by \Cref{Lemma:BadPeriods} and \Cref{Lemma:LengthComp}, the word $\tau(g(f^{n-1}((\tt{02201})^{12/5})))$ has exponent at least $7/3$. 

    $\tt{0101}$: We have already seen that $\mathbf{x}$ contains no $\tt{00}$, no $\tt{11}$, and no $\tt{1010}$, so it suffices to show that the word $\tau(g(f^n(\tt{201012})))$ has a factor of exponent at least $16/7$.  For $n=0$, we check directly that $\tau(g(\tt{201012}))$ contains a factor of exponent at least $16/7$, so we may assume that $n\geq 1$.  We observe that $f(\tt{201012})$ contains the even-$12/5$-power $(\tt{20102})^{12/5}$.  But by \Cref{Lemma:BadPeriods} and \Cref{Lemma:LengthComp}, the word $\tau(g(f^{n-1}((\tt{20102})^{12/5})))$ has exponent at least $7/3$.

    $\tt{022022}$ and $\tt{220220}$: Let $u\in\{\tt{022022},\tt{220220}\}$.  For $n=0,1$, we check directly that for all $a,b\in\{\tt{0,1,2}\}$, the word $\tau(g(f^n(aub)))$ contains a factor of exponent at least $16/7$.  So we may assume that $n\geq 2$.  For all $a,b\in\{\tt{0,1,2}\}$, we check that $f^2(aub)$ contains one of the following even-$31/13$-powers: $(\tt{2010220102010})^{31/13}$, $(\tt{2010201020102})^{31/13}$.  By an argument similar to the one used for $\tt{212}$ above, we see that $\tau(g(f^n(aub)))$ contains a factor of exponent at least $7/3$.
\end{proof}

In \Cref{Lemma:BadCubes} and \Cref{Lemma:ForbiddenRepetitions}, we demonstrated that certain factors in $\mathbf{x}$ lead to long repetitions in $\tau(g(f^n(\mathbf{x})))$.  We now wish to demonstrate that certain factors in $\mathbf{x}$ lead to non-rich factors in $\tau(g(f^n(\mathbf{x})))$.  The maps $f$ and $g$ belong to the well-studied class $P_{\text{ret}}$ (see~\cite{BalkovaPelantovaStarosta2011,DolcePelantova2022}).  In particular, it is known that morphisms in class $P_{\text{ret}}$ preserve non-richness of infinite words.  However, the alternation of $\tau$ complicates things for us here.

We say that a word $w\in\Sigma_3^*$ is \emph{poor} if every palindromic prefix of $w$ with an even number of $\tt{2}$'s occurs at least once more in $w$ with an even number of $\tt{2}$'s before the occurrence.  In other words, $w$ is poor if for every $u\in\Sigma_3^*$, if $u$ is a palindromic prefix of $w$ with an even number of $\tt{2}$'s, then there exist words $p,s\in\Sigma_3^*$ such that $w=pus$, $p\neq \varepsilon$, and $p$ has an even number of $\tt{2}$'s.  The following are examples of poor words:
\begin{itemize}
    \item $\tt{2012}$, whose only palindromic prefix with an even number of $\tt{2}$'s is the empty word, which occurs again as a suffix with an even number of $\tt{2}$'s before it;
    \item $\tt{01220}$, which has two palindromic prefixes with an even number of $\tt{2}$'s, namely $\varepsilon$ and $\tt{0}$, both of which occur as suffixes with an even number of $\tt{2}$'s before them; and 
    \item $\tt{0220102020220}$, which has three palindromic prefixes with an even number of $\tt{2}$'s, namely $\varepsilon$, $\tt{0}$, and $\tt{0220}$, all of which occur as suffixes with an even number of $\tt{2}$'s before them.
\end{itemize}
On the other hand, the word $\tt{0120}$ is neither rich nor poor---the palindromic prefix $\tt{0}$ occurs just once more as a suffix, and there are an \emph{odd} number of $\tt{2}$'s before this second occurrence.  We will prove the following.

\begin{lemma}\label{Lemma:ForbiddenRichness}
    Let $\mathbf{x}\in\Sigma_3^\omega$, and suppose for some $n\geq 0$ that $\tau(g(f^n(\mathbf{x})))$ is rich.  Then $\mathbf{x}$ contains no poor factor.
\end{lemma}

We first prove several intermediate lemmas, for which we make use of one more term.  We say that a word $w\in\Sigma_3^*$ is \emph{middle-class} if it begins in $\tt{2}$ and every odd-length palindromic prefix of $w$ occurs at least once more in $w$ starting at an even index.  In other words, $w$ is middle-class if it begins in $\tt{2}$ and for every word $u\in\Sigma_3^*$, if $u$ is a palindromic prefix of $w$ of odd length, then there exist words $p,s\in\Sigma^*$ such that $w=pus$, $p\neq \varepsilon$, and $p$ has even length.  The following are examples of middle-class words:
\begin{itemize}
    \item $\tt{2202122}$, whose only palindromic prefix of odd length is $2$, which occurs again starting at even index $6$; and
    \item $\tt{202122202}$, whose only palindromic prefixes of odd length are $\tt{2}$ and $\tt{202}$, which both occur again starting at even index $6$.
\end{itemize}
On the other hand, the word $\tt{2012}$ is not middle-class.  It has palindromic prefix $\tt{2}$ of odd length, which occurs again only at odd index $3$.

Roughly speaking, we will show the following:
\begin{itemize}
    \item $f$ sends poor words to poor words;
    \item $g$ sends poor words to middle-class words; and
    \item $\tau$ sends middle-class words to non-rich words.
\end{itemize}
Throughout, we use the following result.

\begin{lemma}\label{Lemma:PalindromicPreimage}
Let $u,w\in\Sigma_3^*$.
\begin{enumerate}[label=\normalfont (\roman*)]
    \item If $f(u)\tt{0}$ is a palindromic prefix of $f(w)\tt{0}$, then $u$ is a palindromic prefix of $w$.
    \item If $g(u)\tt{2}$ is a palindromic prefix of $g(w)\tt{2}$, then $u$ is a palindromic prefix of $w$.
    \item If $u$ starts with $\tt{2}$ and $\tau(u)\tt{00}$ is a palindromic prefix of $\tau(w)\tt{00}$, then $u$ is a palindromic prefix of $w$ of odd length.
\end{enumerate}
\end{lemma}

\begin{proof}
    Parts (i) and (ii) can be proven in a manner similar to~\cite[Lemma 4]{CurrieMolRampersad2020}.  The argument for part (iii) is similar; the alternation of $\tau$ forces $u$ to have odd length.
\end{proof}

\begin{lemma}\label{Lemma:PoorToPoor}
    Let $w\in\Sigma_3^*$.  If $w$ is poor, then $f(w)\tt{0}$ is poor.
\end{lemma}

\begin{proof}
    Let $p$ be a palindromic prefix of $f(w)\tt{0}$ with an even number of $\tt{2}$'s.  Note that $f(w)\tt{0}$ starts with $\tt{0}$.  If $p=\varepsilon$, then $p$ occurs again in $f(w)\tt{0}$ after the prefix $\tt{0}$, which has an even number of $\tt{2}$'s.  So we may assume that $p\neq \varepsilon$.  Since $p$ is a palindromic prefix of $f(w)\tt{0}$, and $f(w)\tt{0}$ starts with $\tt{0}$, we see that $p$ must begin and end in $\tt{0}$.  Hence, we can write $p=f(u)\tt{0}$ for some word $u\in\Sigma_3^*$, and by \Cref{Lemma:PalindromicPreimage}, we see that $u$ is a palindromic prefix of $w$.  Since $p$ has an even number of $\tt{2}$'s, and since $\tt{2}$ is the only letter whose $f$-image has an odd number of $\tt{2}$'s, we conclude that the word $u$ must also contain an even number of $\tt{2}$'s.  Since $w$ is poor, the word $u$ occurs at least once more in $w$ with an even number of $\tt{2}$'s before the occurrence.  It follows that the word $p=f(u)\tt{0}$ occurs at least once more in $f(w)\tt{0}$ with an even number of $\tt{2}$'s before the occurrence, i.e., that $f(w)\tt{0}$ is poor.
\end{proof}

\begin{lemma}\label{Lemma:PoorToMiddle}
    Let $w\in\Sigma_3^*$.  If $w$ is poor, then $g(w)\tt{2}$ is middle-class.
\end{lemma}

\begin{proof}
    First note that $g(w)\tt{2}$ begins in $\tt{2}$.  Let $p$ be an odd-length palindromic prefix of $g(w)\tt{2}$.  Since $p$ is palindromic, it must begin and end with $\tt{2}$, hence we can write $p=g(u)\tt{2}$ for some word $u\in\Sigma_3^*$, and by \Cref{Lemma:PalindromicPreimage}, the word $u$ is a palindromic prefix of $w$.  Since $p$ has odd length, we see that $g(u)$ has even length, and since $\tt{2}$ is the only letter whose $g$-image has odd length, it follows that $u$ has an even number of $\tt{2}$'s.  Since $w$ is poor, the word $u$ occurs at least once more in $w$ with an even number of $\tt{2}$'s before the occurrence.  It follows that $p=g(u)\tt{2}$ occurs at least once more in $g(w)\tt{2}$ starting at an even index.  Therefore, we conclude that the word $g(w)\tt{2}$ is middle-class.
\end{proof}

\begin{lemma}\label{Lemma:MiddleToRich}
    Let $w\in\Sigma_3^*$.  If $w$ is middle-class, then $\tau(w)\tt{00}$ is not rich.
\end{lemma}

\begin{proof}
    Let $p$ be the longest palindromic prefix of $\tau(w)\tt{00}$.  Since $w$ is middle-class, it begins in $\tt{2}$, and we see that $p$ must have prefix $\tau(\tt{2})\tt{00}$.  Since $p$ is palindromic, it must also end in $\tau(\tt{2})\tt{00}$, and not its sister.  Note that every occurrence of $\tau(\tt{2})\tt{00}$ in $\tau(w)\tt{00}$ corresponds in an obvious manner to an occurrence of $\tt{2}$ in $w$ at an even index.  
    So we can write $p=\tau(u)\tt{00}$ for some word $u\in\Sigma_3^*$.  Further, by \Cref{Lemma:PalindromicPreimage}, the word $u$ is a palindromic prefix of $w$ of odd length. 
    Since $w$ is middle-class, the word $u$ occurs at least once more in $w$ starting at an even index.  It follows that $p=\tau(u)\tt{00}$ occurs at least twice in $\tau(w)\tt{00}$.  Since $p$ is the longest palindromic prefix of $w$, it follows that every palindromic prefix of $w$ occurs at least twice in $w$.  By \Cref{Lemma:UnioccurrentSuffix}, we conclude that $w$ is not rich.
\end{proof}

We are now ready to prove \Cref{Lemma:ForbiddenRichness}.

\begin{proof}[Proof of \Cref{Lemma:ForbiddenRichness}]
    Suppose towards a contradiction that $\mathbf{x}$ contains a poor factor $w$.  By a straightforward induction using \Cref{Lemma:PoorToPoor}, we see that the word
    \begin{equation*}
        W_1=f(f(\cdots(f(w)\tt{0})\cdots)\tt{0})\tt{0},
    \end{equation*}
    where $f$ is applied $n$ times, is poor.  Hence, by \Cref{Lemma:PoorToMiddle}, the word $W_2=g(W_1)\tt{2}$ is middle-class, and by \Cref{Lemma:MiddleToRich}, the word $W_3=\tau(W_2)\tt{00}$ is not rich.  But $W_3$ or its sister is a factor of $\tau(g(f^n(\mathbf{x})))$, which contradicts the assumption that $\tau(g(f^n(\mathbf{x})))$ is rich.
\end{proof}

\subsection{The Inner Layers}\label{Subsection:InnerLayers}

\begin{proposition}\label{Prop:InnerLayers}
    Let $\mathbf{x}\in\Sigma_3^\omega$, and suppose for some $n\geq 0$ that $\tau(g(f^n(\mathbf{x})))$ is $16/7$-power-free and rich.  Then a suffix of $\mathbf{x}$ has the form $f(\mathbf{x}')$ for some word $\mathbf{x}'\in\Sigma_3^\omega$.
\end{proposition}

\begin{proof}
First observe the following.
\begin{itemize}
    \item By \Cref{Lemma:BadCubes}, the word $\mathbf{x}$ is $5$-power-free and even-$3$-power-free.
    \item By \Cref{Lemma:ForbiddenRepetitions}, no factor from $F=F_1\cup F_2\cup F_3\cup F_4$ appears internally in $\mathbf{x}$.
    \item By \Cref{Lemma:ForbiddenRichness}, the word $\mathbf{x}$ contains no poor factor.
\end{itemize}
So by taking a suffix if necessary, we may assume that $\mathbf{x}$ is $5$-power-free, even-$3$-power-free, and contains neither poor words nor words from $F$ as factors.  We use these properties frequently throughout the remainder of the proof without further reference.

For ease of writing, we consider an extension $\hat{f}$ of $f$ to $\Sigma_8^*$, defined by
\begin{align*}
    \hat{f}(\tt{0})&=\tt{01}, & \hat{f}(\tt{4})&=\tt{0121},\\
    \hat{f}(\tt{1})&=\tt{022}, & \hat{f}(\tt{5})&=\tt{01221},\\
    \hat{f}(\tt{2})&=\tt{02}, & \hat{f}(\tt{6})&=\tt{012},\\
    \hat{f}(\tt{3})&=\tt{0222}, & \hat{f}(\tt{7})&=\tt{021}.
\end{align*}
Observe that for all $a\in\Sigma_3$, we have $\hat{f}(a)=f(a)$, so $\hat{f}$ is indeed an extension of $f$.

\begin{figure}
    \centering
    \begin{tikzpicture}[shorten >=2pt]
        \node[state,align=center] (02) at (0,0) {\footnotesize $\tt{2}$\\ \scriptsize $\tt{02}$};
        \node[state,align=center] (0222) at (0,2) {\footnotesize $\tt{3}$\\ \scriptsize $\tt{0222}$};
        \node[state,align=center] (01) at (2,2) {\footnotesize $\tt{0}$\\ \scriptsize $\tt{01}$};
        \node[state,align=center] (022) at (2,0) {\footnotesize $\tt{1}$\\ \scriptsize $\tt{022}$};
        \node[state,align=center] (012) at (4,2) {\footnotesize $\tt{6}$\\ \scriptsize $\tt{012}$};
        \node[state,align=center] (021) at (4,0) {\footnotesize $\tt{7}$\\ \scriptsize $\tt{021}$};
        \node[state,align=center] (01221) at (6,1) {\footnotesize $\tt{5}$\\ \scriptsize $\tt{01221}$};
        \node[state,align=center] (0121) at (8,1) {\footnotesize $\tt{4}$\\ \scriptsize $\tt{0121}$};
        \path[style=->]
        (02) edge[bend left=10] (01)
        (01) edge[bend left=10] (02)
        (02) edge[bend left=10] (022)
        (022) edge[bend left=10] (02)
        (0222) edge[bend left=10] (01)
        (01) edge[bend left=10] (0222)
        (022) edge[bend left=10] (01)
        (01) edge[bend left=10] (022)
        (012) edge[bend left=10] (021)
        (021) edge[bend left=10] (012)
        (01221) edge[bend left=10] (0121)
        (0121) edge[bend left=10] (01221)
        (012) edge (01)
        (022) edge (021)
        (01) edge (021)
        (012) edge (022)
        (021) edge (01221)
        (01221) edge (012)
        (02) edge[loop left] (02)
        (0222) edge[loop left] (0222)
        (0121) edge[loop right] (0121)
        (01221) edge[loop above] (01221)
        ;
    \end{tikzpicture}   
    \caption{The directed graph $H$ showing possible transitions between letters in $\mathbf{x}'$.}
    \label{Fig:x'graph}
\end{figure}
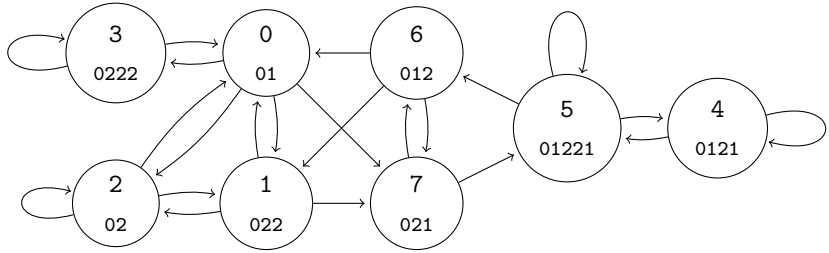

\begin{claim}\label{Claim:fhat}
A suffix of $\mathbf{x}$ has the form $\hat{f}(\mathbf{x}')$ for some word $\mathbf{x}'\in \Sigma_8^\omega$.  Further, the word $\mathbf{x}'$ corresponds to an infinite walk on the graph $H$ drawn in \Cref{Fig:x'graph}.
\end{claim}

\begin{subproof}[Proof of \Cref{Claim:fhat}]
The longest word over $\{\tt{1,2}\}$ that contains no factor from $F_1$ has length $4$, so the letter $\tt{0}$ must occur infinitely many times in $\mathbf{x}$, and by taking a suffix if necessary, we may assume that $\mathbf{x}$ starts with $\tt{0}$.  Thus, it suffices to show that the only possible complete returns to $\tt{0}$ in $\mathbf{x}$ have the form $\hat{f}(a)\tt{0}$ for some $a\in\Sigma_8$.  Consider the tree drawn in \Cref{Fig:fTree}, which shows all possible complete returns to $\tt{0}$ in $\mathbf{x}$.  The words corresponding to red leaves in the tree have a suffix in $F_1$, while the words corresponding to yellow leaves in the tree are poor.  So the words corresponding to green leaves in the tree are the only possible complete returns to $\tt{0}$ in $\mathbf{x}$.  Thus, we can write $\mathbf{x}=\hat{f}(\mathbf{x}')$ for some $\mathbf{x}'\in\Sigma_8^\omega$.

\begin{figure}
        \centering
        \begin{forest}
            [
            $\tt{0}$,for tree={grow=0,l=1.8cm}
                [
               	$\tt{2}$
                    [
                        $\tt{2}$
                        [
                            $\tt{2}$
                            [
                                $\tt{2}$,draw,fill=red
                            ]
                            [
                                $\tt{1}$,draw, fill=red
                            ]
                            [
                                $\tt{0}$,draw,fill=green
                            ]
                        ]
                        [
                            $\tt{1}$
                            [
                                $\tt{2}$,draw,fill=red
                            ]
                            [
                                $\tt{1}$,draw,fill=red
                            ]
                            [
                                $\tt{0}$,draw,fill=yellow
                            ]
                        ]
                        [
                            $\tt{0}$,draw,fill=green
                        ]
                    ]
                    [
                        $\tt{1}$
                        [
                            $\tt{2}$,draw,fill=red
                        ]
                        [
                            $\tt{1}$,draw,fill=red
                        ]
                        [
                            $\tt{0}$,draw,fill=green
                        ]
                    ]
                    [
                        $\tt{0}$,draw,fill=green
                    ]
                ]
                [
            	$\tt{1}$
                    [
                        $\tt{2}$
                        [
                            $\tt{2}$
                            [
                                $\tt{2}$, draw, fill=red
                            ]
                            [
                                $\tt{1}$
                                [
                                    $\tt{2}$,draw,fill=red
                                ]
                                [   
                                    $\tt{1}$,draw,fill=red
                                ]
                                [
                                    $\tt{0}$,draw,fill=green
                                ]
                            ]
                            [
                                $\tt{0}$,draw,fill=yellow
                            ]
                        ]
                        [
                            $\tt{1}$
                            [
                                $\tt{2}$,draw,fill=red
                            ]
                            [
                                $\tt{1}$,draw,fill=red
                            ]
                            [
                                $\tt{0}$,draw,fill=green
                            ]
                        ]
                        [
                            $\tt{0}$,draw,fill=green
                        ]
                    ]
                    [
                        $\tt{1}$,draw,fill=red
                    ]
                    [
                        $\tt{0}$,draw,fill=green
                    ]
                ]
                [
            	$\tt{0}$,draw,fill=red
                ]
            ]
        \end{forest}
        \caption{The tree showing all possible complete returns to $\tt{0}$ in $\mathbf{x}$.}
        \label{Fig:fTree}
    \end{figure}
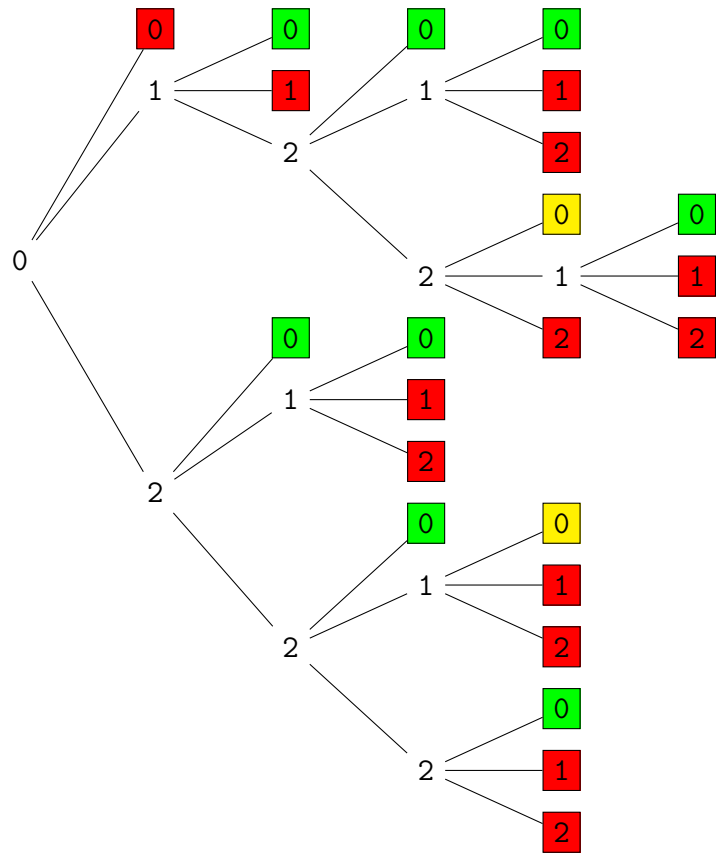

It remains to show that $\mathbf{x}'$ corresponds to an infinite walk on the graph $H$ drawn in \Cref{Fig:x'graph}.  First note that the image of every letter under $\hat{f}$ has prefix $\tt{0}$.  So if $u$ is a factor of $\mathbf{x}'$, then $\hat{f}(u)\tt{0}$ is a factor of $\mathbf{x}$.  We use this fact to show that some letters cannot appear immediately after others in $\mathbf{x}'$.  Consider the letter $\tt{4}$, for example.  It cannot be followed by the letter $\tt{0}$, because $\hat{f}(\tt{40})\tt{0}=\tt{0121010}$ contains the factor $\tt{1010}\in F_1$.  It cannot be followed by a letter from $\{\tt{1,2,3,7}\}$, because for all $a\in\{\tt{1,2,3,7}\}$, the word $\hat{f}(\tt{4}a)$ contains the poor word $\tt{2102}$.  Finally, it cannot be followed by the letter $\tt{6}$, because $\hat{f}(\tt{46})\tt{0}=\tt{01210120}$ is poor.  So the letter $\tt{4}$ can only be followed by $\tt{4}$ or $\tt{5}$ in $\mathbf{x}'$.  By performing a similar analysis on all letters of $\Sigma_8$, ruling out the transition from $a$ to $b$ if the word $\hat{f}(ab)\tt{0}$ contains a poor word or a word in $F_1$, we obtain the directed graph $H$.
\end{subproof}

So $\mathbf{x}$ has the form $\hat{f}(\mathbf{x}')$ for some word $\mathbf{x}'\in\Sigma_8^\omega$.  Our goal now is to show that in fact we have $\mathbf{x}'\in \Sigma_3^\omega$.

\begin{claim}\label{Claim:No6or7}
    The word $\mathbf{x}'$ contains neither $\tt{6}$ nor $\tt{7}$, i.e., we have $\mathbf{x}'\in\Sigma_6^\omega$.
\end{claim}

\begin{subproof}[Proof of \Cref{Claim:No6or7}]
Consider an occurrence of $\tt{6}$ in $\mathbf{x}'$. From the digraph $H$, we see that $\tt{6}$ is followed by a factor of the form $ua$, where $u \in \{\tt{0}, \tt{1}\}^*$ and $a \in \{\tt{2}, \tt{3}, \tt{7}\}$. As $\tt{00}$, $\tt{11}$, $\tt{0101}$, and $\tt{1010}$ are in $F_1$, we see that $u \in \{\varepsilon,\tt{0},\tt{1},\tt{01},\tt{10},\tt{010},\tt{101}\}$. If $a \neq \tt{7}$, then the only palindromic prefixes of $\hat{f}(\tt{6}ua)\tt{0}$ are $\varepsilon$ and $\tt{0}$, because $\hat{f}(\tt{6}ua)\tt{0}$ has prefix $\tt{012}$ and contains no $\tt{21}$. Further, both palindromic prefixes occur as a suffix in $\hat{f}(\tt{6}ua)\tt{0}$ after an even number of $\tt{2}$'s, meaning that $\hat{f}(\tt{6}ua)\tt{0}$ is poor. Thus every time the letter $\tt{6}$ appears in $\mathbf{x}'$, it appears as a prefix of $\tt{6}u\tt{7}$ for some $u \in \{\varepsilon,\tt{0},\tt{1},\tt{01},\tt{10},\tt{010},\tt{101}\}$. 
But for each $u\in\{\tt{1},\tt{01},\tt{10},\tt{010},\tt{101}\}$, we check that $\hat{f}(\tt{6}u\tt{7})$ contains a factor in $F_2$.
So every $\tt{6}$ in $\mathbf{x}'$ occurs as a prefix of either $\tt{67}$ or $\tt{607}$.  

Now consider an occurrence of $\tt{7}$ in $\mathbf{x}'$.  From the digraph $H$, we see that $\tt{7}$ is followed by a factor of the form $v\tt{6}$, where $v\in\{\tt{4},\tt{5}\}^*$.
First note that $\tt{55}$ is not a factor of $v$, since in that case $\hat{f}(\tt{7}v\tt{6})$ contains the factor $\tt{21}\hat{f}(\tt{55})\tt{012}=(\tt{21012})^3$, which is an even-$3$-power.  Next, note that $\tt{4}$ is not a factor of $v$, since in that case $v$ must have suffix $\tt{45}$, and in turn $\hat{f}(\tt{7}v\tt{6})\tt{0}$ has suffix $\hat{f}(\tt{456})\tt{0}$, which is poor.
So we must have $v\in\{\varepsilon,\tt{5}\}$, i.e., every $\tt{7}$ in $\mathbf{x}'$ occurs as a prefix of either $\tt{76}$ or $\tt{756}$.

So suppose that the letter $\tt{7}$ appears in $\mathbf{x}'$.  By taking a suffix of $\mathbf{x}$ if necessary, we may assume that $\mathbf{x}'$ starts with $\tt{7}$.  Then we can write $\mathbf{x}'=\phi(\mathbf{x}'')$ for some $\mathbf{x}''\in\Sigma_4^\omega$,
where
\begin{align*}
    \phi(\tt{0})&=\tt{76},\\
    \phi(\tt{1})&=\tt{760},\\
    \phi(\tt{2})&=\tt{756},\\
    \phi(\tt{3})&=\tt{7560}.
\end{align*}
In turn, we have $\mathbf{x}=\hat{f}(\phi(\mathbf{x}''))$.  Observe that for every letter $a\in\Sigma_4$, the word $\hat{f}(\phi(a))$ has an even number of $\tt{2}$'s.  Since $\mathbf{x}$ is even-$3$-power-free, we see that $\mathbf{x}''$ must be $3$-power-free.  We also claim that $\mathbf{x}''$ contains no factor from the set 
\begin{equation*}
F_\phi=\{\tt{00},\tt{11},\tt{22},\tt{33},\tt{01},\tt{20},\tt{31}\}.
\end{equation*}
Let $w\in F_{\phi}$, and suppose that $w$ appears in $\mathbf{x}''$.  Note that for every letter $a\in\Sigma_4$, the word $\hat{f}(\phi(a))$ has prefix $\hat{f}(\phi(\tt{0}))=\tt{021012}$.  It follows that the word $\hat{f}(\phi(w\tt{0}))$ appears in $\mathbf{x}$.  However, we check that $\hat{f}(\phi(w\tt{0}))$ contains a factor from $F_2$, hence we conclude that $\mathbf{x}''$ contains no factor from $F_\phi$.

Now we run a backtracking algorithm that searches through the tree of all words $u\in\Sigma_4^*$ such that 
\begin{enumerate}[label=(\roman*)]
    \item $u$ is $3$-power-free;
    \item $u$ contains no factor from $F_\phi$; and
    \item $\hat{f}(\phi(u))$ contains no poor factor.
\end{enumerate}  
We find that the longest such word has length $8$.  Thus, we conclude that the letter $\tt{7}$ does not appear in $\mathbf{x}'$.  Since every $\tt{6}$ in $\mathbf{x}'$ appears as a prefix of either $\tt{67}$ or $\tt{607}$, it follows that the letter $\tt{6}$ does not appear in $\mathbf{x}'$ either.
\end{subproof}

\begin{claim}\label{Claim:No4or5}
    The word $\mathbf{x}'$ contains neither $\tt{4}$ nor $\tt{5}$, i.e., we have $\mathbf{x}'\in\Sigma_4^\omega$.
\end{claim}

\begin{subproof}[Proof of \Cref{Claim:No4or5}]
Suppose that $\tt{4}$ or $\tt{5}$ appears in $\mathbf{x}'$.  By \Cref{Claim:No6or7}, we have $\mathbf{x}'\in\Sigma_6^\omega$.  So by considering the digraph $H$, we see that $\mathbf{x}'\in\{\tt{4,5}\}^\omega$.  Since $\mathbf{x}$ contains no $5$-power, we see that $\mathbf{x}'$ contains no $5$-power.  Further, we claim that some suffix of $\mathbf{x}'$ contains no factor from the set $F'=\{\tt{55},\tt{444},\tt{5445}\}$.  Observe that for all $a\in\{\tt{4,5}\}$, the word $\hat{f}(a)$ has prefix $\tt{012}$ and suffix $\tt{21}$.  So if some word $u$ appears internally in $\mathbf{x}'$, then $\mathbf{x}$ contains the word $\tt{21}\hat{f}(u)\tt{012}$. We handle each word in $F'$ separately.
\begin{itemize}
    \item $\tt{55}$: Suppose that $\tt{55}$ appears internally in $\mathbf{x}'$.  It follows that $\mathbf{x}$ contains the word $\tt{21}\hat{f}(\tt{55})\tt{012}=(\tt{21012})^3$.  But this contradicts the fact that $\mathbf{x}$ is even-$3$-power-free.
    \item $\tt{444}$: We first claim that $\tt{4444}$ does not appear internally in $\mathbf{x}$. For if $\tt{4444}$ appears internally in $\mathbf{x}'$, then $\mathbf{x}$ contains the $5$-power $\tt{21}\hat{f}(\tt{4444})\tt{01}= (\tt{2101})^5$; a contradiction.  So by taking a suffix if necessary, we may assume that $\mathbf{x}'$ contains no $\tt{4444}$.  Now suppose that $\tt{444}$ appears internally in $\mathbf{x}'$.  Then $\mathbf{x}'$ contains the factor $\tt{54445}$.  But this is impossible, since $\hat{f}(\tt{54445})$ contains the factor $\tt{2}(\tt{2101})^{17/4}\tt{2}\in F_3$.
    \item $\tt{5445}$: Suppose that $\tt{5445}$ appears in $\mathbf{x}'$ with at least two letters before it.  Then since $\tt{55}$ does not appear internally in $\mathbf{x}'$, we see that $\mathbf{x}'$ must contain an internal occurrence of $\tt{454454}$.  But then $\mathbf{x}$ contains the word $\tt{21}\hat{f}(\tt{454454})\tt{012}=(\tt{2101210122101})^{31/13}\in F_3$, a contradiction.
\end{itemize}

Now we run a backtracking algorithm that searches through the tree of all words $u\in\{\tt{4,5}\}^*$ such that
\begin{enumerate}[label=(\roman*)]
    \item $u$ is $5$-power-free;
    \item $u$ contains no factor from $F'$; and
    \item $\hat{f}(u)$ contains no poor factor.
\end{enumerate}
The longest such word has length $11$.  Thus, we conclude that neither $\tt{4}$ nor $\tt{5}$ appears in $\mathbf{x}'$.
\end{subproof}

\begin{claim}\label{Claim:No3}
    The word $\mathbf{x}'$ contains no $\tt{3}$, i.e., we have $\mathbf{x}'\in\Sigma_3^\omega$.
\end{claim}

\begin{subproof}[Proof of \Cref{Claim:No3}]
Suppose that the letter $\tt{3}$ appears in $\mathbf{x}'$.  From \Cref{Claim:No4or5}, we have $\mathbf{x}'\in\Sigma_4^\omega$.  We claim that the letter $\tt{2}$ does not appear in  $\mathbf{x}'$.  For if it did, then $\mathbf{x}'$ would have a factor of the form $\tt{3}u\tt{2}$ or $\tt{2}u\tt{3}$, where $u\in\{\tt{0,1}\}^*$.  But for all $u\in\{\tt{0,1}\}^*$, the words $\hat{f}(\tt{3}u\tt{2})\tt{0}$ and $\hat{f}(\tt{2}u\tt{3})\tt{0}$ are poor; their only palindromic prefixes with an even number of $\tt{2}$'s are $\varepsilon$ and $\tt{0}$, and they occur again (as suffixes) after an even number of $\tt{2}$'s.  

So we have $\mathbf{x}'\in\{\tt{0,1,3}\}^\omega$.  Observe that $\mathbf{x}'$ does not contain the factor $\tt{3333}$.  For if it did, then we see from the digraph $H$ that $\mathbf{x}'$ would contain either $\tt{33333}$ or $\tt{33330}$.  But $\tt{33333}$ is a $5$-power, and $\hat{f}(\tt{33330})=(\tt{0222})^{17/4}\tt{1}$ is in $F_4$.  By taking a suffix if necessary, assume that $\mathbf{x}'$ starts with $\tt{0}$.  Then by considering the digraph $H$, and remembering that $\mathbf{x}'$ does not contain $\tt{3333}$, we see that we must have $\mathbf{x}'=\psi(\mathbf{x}'')$ for some word $\mathbf{x}''\in\{\tt{0,1,2,3}\}^\omega$, where 
\begin{align*}
    \psi(\tt{0})&=\tt{03},\\
    \psi(\tt{1})&=\tt{033},\\
    \psi(\tt{2})&=\tt{0333},\\
    \psi(\tt{3})&=\tt{01}.
\end{align*}
In turn, we have $\mathbf{x}=\hat{f}(\psi(\mathbf{x}''))$.  Since $\mathbf{x}$ contains no $5$-power, we see that $\mathbf{x}''$ contains no $5$-power.
Note that for all $a\in\Sigma_4$, the word $\hat{f}(\psi(a))$ has prefix $\tt{01022}$ and suffix $\tt{22}$.  Further, for all $a\in \Sigma_3$, the word $\hat{f}(\psi(a))$ has suffix $\tt{0222}$, and for all $v\in\Sigma_3^*$ with $|v|\geq 2$, the word $\hat{f}(\psi(v))$ has prefix $\tt{0102220}$.

Now we claim that some suffix of $\mathbf{x}''$ belongs to $\Sigma_2^\omega$.  Suppose first that the letter $\tt{3}$ appears internally in $\mathbf{x}''$.  Then $\mathbf{x}$ contains the word $\tt{22}\hat{f}(\psi(\tt{3}))\tt{01022}=(\tt{22010})^{12/5}\in F_4$, a contradiction.  So by taking a suffix if necessary, we may assume that $\mathbf{x}''\in\Sigma_3^\omega$.  Suppose now that the letter $\tt{2}$ appears in $\mathbf{x}''$.  If $\tt{20}$ appears in $\mathbf{x}''$, then $\mathbf{x}$ contains the poor factor $\hat{f}(\psi(\tt{20}))\tt{010}$, a contradiction.  If $\tt{21}$ appears in $\mathbf{x}''$, then $\mathbf{x}$ contains the word $\hat{f}(\psi(\tt{21}))\tt{010222}$, which has suffix $(\tt{0222022201})^{12/5}\in F_4$,
a contradiction.  So every occurrence of $\tt{2}$ in $\mathbf{x}''$ must be followed by another $\tt{2}$.  But then $\mathbf{x}''$ contains the $5$-power $\tt{22222}$, a contradiction.

So we may assume that $\mathbf{x}''\in\Sigma_2^\omega$.  We now claim that some suffix of $\mathbf{x}''$ contains no factor from the set 
\begin{equation*}
F_{\psi}=\{\tt{11},\tt{000},\tt{1001}\}.
\end{equation*}
We handle each word in $F_\psi$ separately.
\begin{itemize}
    \item $\tt{11}$: Suppose that $\tt{11}$ appears internally in $\mathbf{x}''$. Then $\mathbf{x}$ contains the word 
    \[
    \tt{0222}\hat{f}(\psi(\tt{11}))=(\tt{0222010222})^{12/5}\in F_4,
    \]
    a contradiction.
    \item $\tt{000}$: Suppose that $\tt{000}$ appears internally in $\mathbf{x}''$.  Then $\mathbf{x}$ contains the word 
    \[
    \tt{0222}\hat{f}(\psi(\tt{000}))\tt{0102220}=(\tt{022201})^{29/6}\in F_4,
    \]
    a contradiction.
    \item $\tt{1001}$: Suppose that $\tt{1001}$ appears with at least two letters before it in $\mathbf{x}''$.  Since $\tt{11}$ does not appear internally in $\mathbf{x}''$, we see that the word $\tt{010010}$ appears internally in $\mathbf{x}''$.  But then $\mathbf{x}$ contains the word \[
    \tt{0222}\hat{f}(\psi(\tt{010010}))\tt{0102220}=(\tt{0222010222010222022201})^{5/2}\in F_4,
    \]
    a contradiction.
\end{itemize}

Now we run a backtracking algorithm that searches through the tree of all words $u\in\Sigma_2^*$ such that
\begin{enumerate}[label=(\roman*)]
    \item $u$ is $5$-power-free;
    \item $u$ contains no factor from $F_{\psi}$; and
    \item $\hat{f}(\psi(u))$ contains no poor factor.
\end{enumerate}
The longest such word has length $11$.  Thus, we conclude that $\tt{3}$ does not appear in $\mathbf{x}'$.
\end{subproof}

Since $\mathbf{x}'\in\Sigma_3^\omega$, we have $\mathbf{x}=\hat{f}(\mathbf{x}')=f(\mathbf{x}')$, and this completes the proof of the proposition.
\end{proof}

\begin{proof}[Proof of \Cref{Theorem:Structure}]
    We proceed by induction on $n$.  As argued at the beginning of \Cref{Subsection:FirstLayer}, by permuting the letters and taking a suffix if necessary, we may assume that $\mathbf{z}$ has prefix $\tt{001002}$.  Then by \Cref{Prop:FirstLayer}, we have $\mathbf{z}=\tau(\mathbf{y})$ for some word $\mathbf{y}\in\Sigma_3^\omega$, and by \Cref{Prop:SecondLayer}, some suffix of $\mathbf{y}$ has the form $g(\mathbf{x}_0)$ for some word $\mathbf{x}_0\in\Sigma_3^\omega$.  This completes the base case.
    
    Now suppose for some integer $k\geq 0$ that some suffix of $\mathbf{z}$ has the form $\tau(g(f^k(\mathbf{x}_k)))$ for some word $\mathbf{x}_k\in\Sigma_3^\omega$.  Then by \Cref{Prop:InnerLayers}, some suffix of $\mathbf{x}_k$ has the form $f(\mathbf{x}_{k+1})$ for some word $\mathbf{x}_{k+1}\in \Sigma_3^\omega$.  It follows that some suffix of $\mathbf{z}$ has the form $\tau(g(f^{k+1}(\mathbf{x}_{k+1})))$.  Therefore, we conclude that the theorem statement holds by mathematical induction.
\end{proof}

\section{Richness}\label{Section:Richness}

Throughout this section, let $\mathbf{x}=f^\omega(\tt{0})$, $\mathbf{y}=g(\mathbf{x})$, and $\mathbf{z}=\tau(\mathbf{y})$. 
For an infinite word $\mathbf{u}$, we let $C_\textbf{u}\colon\mathbb{N}\rightarrow\mathbb{N}$ denote the \emph{factor complexity} of $\mathbf{u}$, and we let $P_\textbf{u}\colon \mathbb{N}\rightarrow\mathbb{N}$ denote the \emph{palindromic complexity} of $\textbf{u}$.  That is, $C_\mathbf{u}(n)$ is the number of distinct factors of $\mathbf{u}$ of length $n$ for all $n\geq 0$, and $P_\mathbf{u}(n)$ is the number of distinct palindromic factors of $\mathbf{u}$ of length $n$ for all $n\geq 0$. 

The goal of this section is to prove the following result.

\begin{proposition}\label{Prop:Rich}
    The word $\mathbf{z}$ is rich.
\end{proposition}

Along the way, we will also show that $\mathbf{x}$ and $\mathbf{y}$ are rich.  In \Cref{SubSection:FactorComplexity} and \Cref{SubSection:PalindromicComplexity}, we determine the factor complexity and the palindromic complexity, respectively, of the word $\mathbf{z}$. \Cref{Prop:Rich} follows immediately from these results and the following result that characterizes infinite rich words in terms of their factor complexity and palindromic complexity functions. Notice that the language of $\mathbf{z}$ is closed under reversal as $\mathbf{z}$ is uniformly recurrent and contains arbitrarily long palindromes (see the results of \Cref{SubSection:PalindromicComplexity}).

\begin{proposition}\cite[Thm.~1.1]{BucciDeLucaGlenZamboni2009}
\label{Prop:RichCharacterization}
Let $\mathbf{u}$ be an infinite word whose language is closed under reversal. Then $\mathbf{u}$ is rich if and only if
\begin{equation}\label{Eq:RichCharacterization}
    P_\textbf{u}(n) + P_\textbf{u}(n+1) = C_\textbf{u}(n+1) - C_\textbf{u}(n) + 2
\end{equation}
for all $n \geq 0$.
\end{proposition}

Before we proceed with determining the factor complexity and palindromic complexity functions of $\mathbf{z}$, we introduce some terminology and prove a preliminary result.  Let $\mathbf{u}$ be an infinite word, and let $w\in\Fact(\mathbf{u})$.  A \emph{left extension} (resp.\ \emph{right extension}) of $w$ in $\mathbf{u}$ is a word of the form $aw\in\Fact(\mathbf{u})$ (resp.\ $wa\in\Fact(\mathbf{u})$), where $a$ is a letter.  A \emph{bi-extension} of $w$ in $\mathbf{u}$ is a word of the form $awb\in Fact(\mathbf{u})$, where $a$ and $b$ are letters.  If $w$ is a palindrome, then a \emph{palindromic extension} of $w$ in $\mathbf{u}$ is a bi-extension of $w$ that is a palindrome.  We say that $w$ is \emph{left-special} (resp.\ \emph{right-special}) in $\mathbf{u}$ if it has at least two distinct left (resp.\ right) extensions, and we say that $w$ is \emph{bispecial} in $\mathbf{u}$ if it is both left-special and right-special.

\begin{lemma}\label{Lemma:SisterOccurs}
  If $z$ is a factor of $\mathbf{z}$, then its sister is a factor of $\mathbf{z}$.
\end{lemma}
\begin{proof}
  We first claim that for each
  $x \in \Fact(\mathbf{x})$, there exist $u$ and $u'$ such that both $ux$ and $u'x$ are prefixes of $\mathbf{x}$, and $|u|_\tt{2}$ and $|u'|_\tt{2}$ have different parity.
  Let $x \in \Fact(\mathbf{x})$.  Then $x$ occurs in $f^n(\tt{0})$ for some $n \geq 0$, i.e., there is some $n\geq 0$ such that $f^n(\tt{0})$ has prefix $ux$ for some word $u$.  Set $v_n = f^{n+2}(\tt{0}) f^{n+1}(\tt{02}) f^n(\tt{0})$. It is not difficult to see that $v_n$ is a prefix of $f^{n+3}(\tt{0})$, and hence of $\mathbf{x}$.  Letting $u'=f^{n+2}(\tt{0})f^{n+1}(\tt{02})u$, we see that $v_n$ (and hence $\mathbf{x}$) has both $ux$ and $u'x$ as prefixes. 
  Note that $|f^k(\tt{0})|_\tt{2}$ is even for all $k\geq 0$ and $|f^k(\tt{2})|_\tt{2}$ is odd for all $k \geq 0$.
  So $|f^{n+2}(\tt{0}) f^{n+1}(\tt{02})|_2$ is odd, and we conclude that $|u|_\tt{2}$ and $|u'|_\tt{2}$ have different parity.

  Now let $y\in \Fact(\mathbf{y})$.  Then there exists a factor $x \in \Fact(\mathbf{x})$ such that $g(x) = \alpha y \beta$ for some words $\alpha$ and $\beta$. Let $u$ and
  $u'$ be as above, so that $g(u)\alpha y$ and $g(u')\alpha y$ are prefixes of $\mathbf{y}$. Since $|g(v)|$ is even if and only if $|v|_\tt{2}$ is even, it follows that the lengths of $g(u)\alpha$ and $g(u')\alpha$ have different parities.  So $y$ occurs starting at both an even position and an odd position in $\mathbf{y}$.

  Finally, let $z\in \Fact(\mathbf{z})$.  Then $z$ occurs in $\tau(y)$ for some prefix $y$ of $\mathbf{y}$.  From above, we see that $y$ must also occur starting at an odd position in $\mathbf{y}$, hence $\overline{\tau}(y)$ also occurs in $\mathbf{z}$, and $\overline{\tau}(y)$ contains the sister of $z$.
\end{proof}

\subsection{Factor complexity}\label{SubSection:FactorComplexity}

\begin{proposition}\label{Prop:xFactorComplexity}
    For all $n \geq 1$, the word $\mathbf{x} = f^\omega(\tt{0})$ has exactly two right-special factors $u$ and $v$ of length $n$. The word $u$ ends with $\tt{0}$ and $u\tt{1}, u\tt{2} \in \Fact(\mathbf{x})$, and the word $v$ ends with $\tt{2}$ and $v\tt{0}, v\tt{2} \in \Fact(\mathbf{x})$.
\end{proposition}
\begin{proof}
    Since $\tt{00},\tt{11},\tt{12},\tt{21}\not\in\Fact(\mathbf{x})$, every non-empty factor of $\mathbf{x}$ has at most two right extensions in $\Fact(\mathbf{x})$. Moreover, there is no right-special factor ending in $\tt{1}$ since neither $\tt{11}$ nor $\tt{12}$ belongs to $\Fact(\mathbf{x})$. Both $\tt{0}$ and $\tt{2}$ are right-special as $\tt{01}$, $\tt{02}$, $\tt{20}$, and $\tt{22}$ all belong to $\Fact(\mathbf{x})$. 

    We claim that $\mathbf{x}$ contains arbitrarily long right-special factors ending with $\tt{0}$ and arbitrarily long right-special factors ending with $\tt{2}$. Let $u$ and $v$ be right-special factors of length $n$ ending respectively with $\tt{0}$ and $\tt{2}$. Such factors exist when $n = 1$. The word $u$ has right extensions $u\tt{1}$ and $u\tt{2}$, hence $f(u)f(\tt{1})$ and $f(u)f(\tt{2})\tt{0}$ are also factors of $\mathbf{x}$. As $f(u)f(\tt{1}) = f(u)f(\tt{2})\tt{2}$, we see that $f(u)f(\tt{2})$ is right-special, ends in $\tt{2}$, and has length greater than $n$. Similarly, $v$ has right extensions $v \tt{0}$ and $v \tt{2}$, and we find that $f(v) \tt{0}$ is right-special, ends in $\tt{0}$, and has length greater than $n$. This proves the claim.
    
    It suffices to show that $\mathbf{x}$ never has three right-special factors of the same length. Suppose, towards a contradiction, that $\mathbf{x}$ has three right-special factors of the same length.  Let $n$ be the least integer such that $\mathbf{x}$ has three right-special factors of length $n$.  By enumerating all right-special factors of $\mathbf{x}$ of length at most $3$, we see that $n \geq 4$.  Now each right-special factor of length $n$ contains a right-special factor of length $n-1$ as a suffix.  By the minimality of $n$, there are only two right-special factors of length $n-1$.  Thus, two distinct right-special factors of length $n$ have a common suffix $u$ of length $n-1$, i.e., they have the form $au$ and $bu$, where $a,b\in\Sigma_3$ and $|u| = n - 1$.  So $u$ is bispecial, meaning that $u$ does not begin (or end) in $\tt{1}$ or $\tt{22}$.  Hence, by deleting at most one letter at the beginning of $u$ (if $\tt{20}$ is a prefix of $u$) and at most two letters at the end of $u$ (if $\tt{02}$ is a suffix of $u$), we obtain a factor $u'=f(v)$ of $u$ for some $v$.  Since $|u|\geq 4$, we have $1\leq |v|<|u|$.  Further, we see that $v$ has four distinct bi-extensions in $\Fact(\mathbf{x})$.  Let $cv$ and $dv$ be the left extensions of $v$ in $\Fact(\mathbf{x})$.  Then $cv$ and $dv$ are right-special and end in the same letter. Thus $\mathbf{x}$ has at least three distinct right-special factors of length $|cv|<|au|=n$, and this contradicts the minimality of $n$.
    %
\end{proof}

\begin{proposition}\label{Prop:yFactorComplexity}
    For all $n \geq 2$, the word $\mathbf{y} = g(f^\omega(\tt{0}))$ has exactly two right-special factors $u$ and $v$ of length $n$. The word $u$ ends with $\tt{02}$ and $u\tt{1}, u\tt{2} \in \Fact(\mathbf{y})$, and the word $v$ ends with $\tt{22}$ and $v\tt{0}, v\tt{2} \in \Fact(\mathbf{y})$. 
\end{proposition}
\begin{proof}
  By considering the images under $g$ of the right extensions of the right-special factors of $\mathbf{x}$ described in \Cref{Prop:xFactorComplexity}, we see that the claimed words $u$ and $v$ exist for all $n\geq 2$. Thus it suffices to show that $\mathbf{y}$ has at most two right-special factors of length $n$ for all $n \geq 2$. The claim is easily checked for $n = 2$.

  Assume for a contradiction that there exists a least integer $n \geq 3$ such that there are three right-special factors of length $n$. As in the proof of \Cref{Prop:xFactorComplexity}, there exists a right-special factor $u$ of length $n - 1$ such that $au$ and $bu$ are right-special for distinct letters $a$ and $b$. 
  Since $\tt{2}$ is the unique left-special letter and the unique right-special letter, we see that $u$ begins and ends in $\tt{2}$. Therefore $u = g(v) \tt{2}$ for a factor $v$ of $\mathbf{x}$. Since $au$ and $bu$ both have two right extensions in $\mathbf{y}$, we deduce that $\mathbf{x}$ has two right-special factors of length $|v| + 1$ that end in the same letter. This contradicts \Cref{Prop:xFactorComplexity}.
\end{proof}

\begin{proposition}\label{Prop:zFactorComplexity}
    For all $n \geq 4$, the word $\mathbf{z} = \tau(g(f^\omega(\tt{0})))$ has $4$ right-special factors of length $n$, each with exactly two right extensions. 
\end{proposition}
\begin{proof}
  This proof is essentially the same as for $\mathbf{x}$ and $\mathbf{y}$, but the presence of $\tau$ complicates the analysis slightly. It is straightforward to verify the statement for $n\leq 9$.
  
  Let $s$ be a right-special factor of $\mathbf{z}$ such that $|s| \geq 10$.  First note that $s$ contains at least one of the letters $\tt{1}$ and $\tt{2}$.  We assume that the letter $\tt{1}$ occurs closer to the end of the word $s$ than the letter $\tt{2}$, i.e., that $s$ has suffix $\tt{1}$, $\tt{10}$, or $\tt{100}$.  (A symmetric argument applies if $s$ has suffix $\tt{2}$, $\tt{20}$, or $\tt{200}$.)  We first argue that $s$ has suffix $\tt{10}$ or $\tt{100}$.  Suppose otherwise that $s$ has suffix $\tt{1}$.  Then we must have $s \tt{0}, s \tt{1}\in \Fact(\mathbf{z})$, since $\tt{12}\not\in \Fact(\mathbf{z})$. The suffix $\tt{11}$ of $s \tt{1}$ forces $s$ to have suffix $\tt{00101}$, but then $s \tt{0}$ has suffix $\tt{001010}$.  Since $\tt{001010}\not\in\Fact(\mathbf{z})$, this is a contradiction.  So $s$ has suffix $\tt{10}$ or $\tt{100}$.  If $s$ has suffix $\tt{10}$, then we have $s \tt{0}, s \tt{1} \in \Fact(\mathbf{z})$ and $s\tt{2}\not\in \Fact(\mathbf{z})$, since $\tt{102} \notin \Fact(\mathbf{z})$. If $s$ has suffix $\tt{100}$, then we have $s \tt{1}, s \tt{2} \in \Fact(\mathbf{z})$ and $s\tt{0}\not\in \Fact(\mathbf{z})$, since $\tt{000}\not\in \Fact(\mathbf{z})$.  In particular, we see that every right-special factor of $\mathbf{z}$ of length at least $10$ has exactly two right extensions.

  Now it suffices to show that $\mathbf{z}$ contains a unique right-special factor of length $n$ with suffix $\tt{10}$ and a unique right-special factor of length $n$ with suffix $\tt{100}$ for each $n\geq 10$, as \Cref{Lemma:SisterOccurs} guarantees that their sisters will also be right-special factors of $\mathbf{z}$.  To see that at least one right-special of length $n$ of each type exists, we consider the images under $\tau$ (or $\overline{\tau}$) of the right extensions of the right-special factors of $\mathbf{y}$ described in \Cref{Prop:yFactorComplexity}.

  Suppose for a contradiction that there is a least integer $n \geq 10$ such that there exist two right-special factors of length $n$ with suffix $\tt{10}$. By the minimality of $n$, they have the form $au$ and $bu$ for distinct letters $a$, $b$ and a right-special factor $u$ of length $n - 1$, and their right extensions are $au\tt{0}$, $au\tt{1}$, $bu\tt{0}$, and $bu\tt{1}$. Since $\tt{1100}, \tt{10101} \notin \Fact(\mathbf{z})$, it must be that $u$ has suffix $\tt{0010}$. Further, $u$ cannot have suffix $\tt{10010}$ as this factor is not right-special. Therefore $u$ has suffix $\tt{20010}$. 
  Since $\Fact(\mathbf{z})$ is closed under reversal and $u$ is left-special, we see from our work above that $u$ has prefix $\tt{0}$.
  
  First assume that $a = \tt{0}$.  Then $b\in\{\tt{1},\tt{2}\}$, and we let $c$ be the other letter in $\{\tt{1},\tt{2}\}$. Since $a\tt{00},b \tt{0} c \notin \Fact(\mathbf{z})$, it must be that $u$ has prefix $\tt{0}b$. Further, as $\tt{00} bb, b \tt{0} b \tt{0} b \notin \Fact(\mathbf{z})$, we see that $u$ must begin with $0b00c$. It follows that $u = \tt{0} b \tt{00} c u' \tt{20010}$ for some word $u'$. Now $\tt{00}c u' \tt{2}$ must equal $\tau(v)$ (if $c = \tt{1}$) or $\overline{\tau}(v)$ (if $c = \tt{2}$) for a nonempty factor $v$ of $\mathbf{y}$.  Further, since $au\tt{0}\in \Fact(\mathbf{z})$, we must have $\tt{0}v\tt{0}\in \Fact(\mathbf{y})$.  Since $au\tt{1}$ is also in $\Fact(\mathbf{z})$, we must have $\tt{0}v\tt{1}$ or $\tt{0}v\tt{2}$ in $\Fact(\mathbf{y})$ as well.  So $\tt{0}v$ is right-special in $\mathbf{y}$, and from \Cref{Prop:yFactorComplexity}, it must be $\tt{0}v\tt{2}\in \Fact(\mathbf{y})$.  By considering the right extensions $bu\tt{0}$ and $bu\tt{1}$ of $bu$, and using the fact that $\Fact(\mathbf{z})$ is closed under reversal, we see that $\tt{2}v\tt{0}$ and $\tt{2}v\tt{2}$ must also be factors of $\mathbf{y}$.  So $\tt{0}v$ and $\tt{2}v$ are distinct right-special factors of $\mathbf{y}$, both with right extensions by $\tt{0}$ and $\tt{2}$ belonging to $\Fact(\mathbf{y})$, and this contradicts \Cref{Prop:yFactorComplexity}.

  We may now assume that $a, b \neq \tt{0}$.  In this case, we see that $u = u' \tt{0010}$ with $u' = \tau(\tt{1}v)$ or $u' = \overline{\tau}(\tt{1}v)$ for a factor $\tt{1}v$ of $\mathbf{y}$.  By an argument similar to the one in the previous paragraph, one can show that $\tt{1}v$ and $\tt{2}v$ are right-special in $\mathbf{y}$, and derive a contradiction with \Cref{Prop:yFactorComplexity}.  Thus we have shown that $\mathbf{z}$ contains a unique right-special factor of length $n$ with suffix $\tt{10}$ for all $n \geq 10$.

  The proof that $\mathbf{z}$ contains a unique right-special factor of length $n$ with suffix $\tt{100}$ for all $n\geq 10$ is similar, and is omitted.
\end{proof}

Note that it now follows easily from \Cref{Prop:xFactorComplexity}, \Cref{Prop:yFactorComplexity}, and \Cref{Prop:zFactorComplexity}, along with a determination of initial values by inspection, that $C_\mathbf{x}(n)=C_\mathbf{y}(n)=2n+1$ for all $n\geq 0$, and that $C_\mathbf{z}(n)=4n+2$ for all $n\geq 4$, with $C_\mathbf{z}(0)=1$, $C_\mathbf{z}(1)=3$, $C_\mathbf{z}(2)=7$, and $C_\mathbf{z}(3)=12$.

We remark that \cite[Corollary~1.4]{BalkovaPelantovaStarosta2009} states that an infinite word with factor complexity $2n + 1$ for all $n$ is rich provided that its language is closed under reversal.  Since $\mathbf{x}$ and $\mathbf{y}$ are uniformly recurrent and contain arbitrarily long palindromes (see the results of \Cref{SubSection:PalindromicComplexity}), their languages are closed under reversal, and we immediately obtain the following.

\begin{corollary}
    The words $\mathbf{x}$ and $\mathbf{y}$ are rich.
\end{corollary}

We note, however, that we cannot deduce that $\mathbf{z}$ is rich directly from its factor complexity.  So we proceed with the determination of the palindromic complexity of $\mathbf{z}$ in the next subsection.

\subsection{Palindromic Complexity}\label{SubSection:PalindromicComplexity}

In order to determine the palindromic complexity of $\mathbf{z}$, we first show that all palindromic factors in $\mathbf{x}$ and $\mathbf{y}$ have unique palindromic extensions.  
Note that this property does not hold for all words of factor complexity $2n + 1$ (c.f.~\cite[Section~4]{BalkovaPelantovaStarosta2009}).

\begin{proposition}\label{Prop:xPalindromicComplexity}
Every palindromic factor of $\mathbf{x}$ has a unique palindromic extension in $\Fact(\mathbf{x})$.
\end{proposition}

\begin{proof}
    We proceed by induction on the length of the palindromic factor.  First, we check that every palindromic factor of $\mathbf{x}$ of length at most $2$ has a unique palindromic extension.

    Now suppose for some $n\geq 3$ that every palindromic factor of $\mathbf{x}$ of length less than $n$ has a unique palindromic extension, and let $w$ be a palindromic factor of $\mathbf{x}$ of length $n$.  Since $w$ is a palindrome of length at least $3$, we have $w=aw'a$ for some letter $a\in \Sigma_3$ and word $w'\in \Sigma_3^*$ with $|w'|\geq 1$.

    \smallskip

    \noindent
    \textbf{Case 1:} Suppose that $w=\tt{0}w'\tt{0}$.  Then $w=f(v)\tt{0}$ for some word $v\in\Fact(\mathbf{x})$ of length less than $w$.  By \Cref{Lemma:PalindromicPreimage}, the word $v$ is a palindrome, and since $|v|<|w|$, the word $v$ has a unique palindromic extension in $\Fact(\mathbf{x})$, say $ava$, where $a\in\Sigma_3$.  Then $\mathbf{x}$ contains the factor $f(ava)$, which contains a palindromic extension of $w$.  Now suppose (towards a contradiction) that there are two distinct palindromic extensions of $w$ in $\Fact(\mathbf{x})$.  Since $\tt{00}\not\in\Fact(\mathbf{x})$, we must have $\tt{1}w\tt{1}=\tt{10}w'\tt{01}$ and $\tt{2}w\tt{2}=\tt{20}w'\tt{02}$ in $\Fact(\mathbf{x})$.  Note that the only possible preimage of $\tt{1}w\tt{1}$ is $\tt{0}v\tt{0}$, while the possible preimages of $\tt{2}w\tt{2}$ are $\tt{1}v\tt{1}$, $\tt{1}v\tt{2}$, $\tt{2}v\tt{1}$, and $\tt{2}v\tt{2}$.  So $\Fact(\mathbf{x})$ must contain $\tt{0}v\tt{0}$ and at least one word from the set 
    $\{\tt{1}v\tt{1},\tt{1}v\tt{2},\tt{2}v\tt{1},\tt{2}v\tt{2}\}$.  Since $v$ has a unique palindromic extension in $\Fact(\mathbf{x})$ by the inductive hypothesis, we see that $\Fact(\mathbf{x})$ must contain $\tt{0}v\tt{0}$ and either $\tt{1}v\tt{2}$ or $\tt{2}v\tt{1}$.  Since $\tt{0}v\tt{0}\in\Fact(\mathbf{x})$ and $\tt{00}\not\in\Fact(\mathbf{x})$, we see that the first and last letter of $v$ is either a $\tt{1}$ or a $\tt{2}$.  But this is impossible, since $\tt{1}v\tt{2}$ or $\tt{2}v\tt{1}$ is also in $\Fact(\mathbf{x})$, and no word from $\{\tt{11},\tt{12},\tt{21}\}$ belongs to $\Fact(\mathbf{x})$.

    \smallskip

    \noindent
    \textbf{Case 2:} Suppose that $w=\tt{1}w'\tt{1}$.  Since every occurrence of $\tt{1}$ in $\mathbf{x}$ is preceded and followed by $\tt{0}$, we see that $\tt{0}w\tt{0}$ is the unique palindromic extension of $w$ in $\Fact(\mathbf{x})$.

    \smallskip
    
    \noindent
    \textbf{Case 3:} Suppose that $w=\tt{2}w'\tt{2}$.  Since $\tt{21}\not\in\Fact(\mathbf{x})$, we see that $w'$ must begin in $\tt{0}$ or $\tt{2}$.  If $w'$ begins in $\tt{2}$, then we must have $|w'|\geq 2$, since $\tt{222}\not\in\Fact(\mathbf{x})$.  So we can write $w=\tt{22}w''\tt{22}$ for some $w''\in\Sigma_3^*$, and since every occurrence of $\tt{22}$ in $\mathbf{x}$ is preceded and followed by $\tt{0}$, we see that the unique palindromic extension of $w$ in $\Fact(\mathbf{x})$ is $\tt{0}w\tt{0}$.  So we may assume that $w'$ begins in $\tt{0}$.  Hence we can write $w'=f(v)\tt{0}$ for some word $v\in\Sigma_3^*$, and by \Cref{Lemma:PalindromicPreimage}, the word $v$ is a palindrome.  Since $|v|<|w'|<|w|$, the words $v$ and $w'$ have a unique palindromic extension in $\Fact(\mathbf{x})$.  So $w=2w'2$ must be the unique palindromic extension of $w'$ in $\Fact(\mathbf{x})$.  It follows that $\tt{0}v\tt{0}\not\in \Fact(\mathbf{x})$, for $f(\tt{0}v\tt{0})$ contains the palindromic extension $1w'1$ of $w'$. 
    Now, if the unique palindromic extension of $v$ in $\Fact(\mathbf{x})$ is $\tt{1}v\tt{1}$, then $\mathbf{x}$ contains the factor $f(\tt{1}v\tt{1})$, which contains the palindromic extension $\tt{2}w\tt{2}$.  If the unique palindromic extension of $v$ in $\Fact(\mathbf{x})$ is $\tt{2}v\tt{2}$, then $\mathbf{x}$ contains the palindromic extension $f(\tt{2}v\tt{2})\tt{0}=\tt{0}w\tt{0}$ of $w$.  Finally, if both $\tt{0}w\tt{0}$ and $\tt{2}w\tt{2}$ are in $\Fact(\mathbf{x})$, then both $\tt{2}v\tt{2}$ and $\tt{1}v\tt{1}$ would be in $\Fact(\mathbf{x})$, a contradiction.

    \smallskip

    By mathematical induction, we conclude that every palindromic factor of $\mathbf{x}$ has a unique palindromic extension.
\end{proof}

\begin{proposition}\label{Prop:yPalindromicComplexity}
Every palindromic factor of $\mathbf{y}$ has a unique palindromic extension in $\Fact(\mathbf{y})$.
\end{proposition}

\begin{proof}
    Let $w$ be a palindromic factor of $\mathbf{y}$.  First note that if $w=\varepsilon$, then the unique palindromic extension of $w$ is $\tt{22}$.  Next note that every occurrence of $\tt{0}$ or $\tt{1}$ in $\mathbf{y}$ is preceded and followed by $\tt{2}$.  So if $w$ begins in $\tt{0}$ or $\tt{1}$, then the unique palindromic extension of $w$ in $\Fact(\mathbf{y})$ is $\tt{2}w\tt{2}$.

    So we may assume that $w$ begins in $\tt{2}$.  By inspecting all factors of $\mathbf{y}$ of length $3$, we see that the unique palindromic extension of $\tt{2}$ is $\tt{222}$.  So we may assume that $|w|\geq 2$, which means that we can write $w=\tt{2}w'\tt{2}$ for some word $w'\in \Sigma_3^*$.  But then we have $w=g(v)\tt{2}$ for some nonempty word $v\in \Fact(\mathbf{x})$, and by \Cref{Lemma:PalindromicPreimage}, the word $v$ is a palindrome.  By \Cref{Prop:xPalindromicComplexity}, the word $v$ has a unique palindromic extension in $\Fact(\mathbf{x})$.  Observe that for all $a\in\Sigma_3^*$, we have $ava\in \Fact(\mathbf{x})$ if and only if $g(ava)\tt{2}\in \Fact(\mathbf{y})$, and that $g(ava)\tt{2}$ contains the palindromic extension $awa$ of $w$.  So if $ava$ is the unique palindromic extension of $v$ in $\Fact(\mathbf{x})$, then $awa$ is the unique palindromic extension of $w$ in $\Fact(\mathbf{y})$.  
\end{proof}

\begin{lemma}\label{Lemma:0w0}
    Let $w$ be a palindromic factor of $\mathbf{y}$.  If $\tt{0}w\tt{0}\in \Fact(\mathbf{y})$, then no word from the set $\{\tt{1}w\tt{1},\tt{1}w\tt{2},\tt{2}w\tt{1},\tt{2}w\tt{2}\}$ is in $\Fact(\mathbf{y})$.
\end{lemma}

\begin{proof}
    Suppose that $\tt{0}w\tt{0}\in\Fact(\mathbf{y})$.  By \Cref{Prop:yPalindromicComplexity}, the word $w$ has a unique palindromic extension in $\Fact(\mathbf{y})$, so we see that neither $\tt{1}w\tt{1}$ nor $\tt{2}w\tt{2}$ belongs to $\Fact(\mathbf{y})$.  It remains to show that $\tt{1}w\tt{2}$ and $\tt{2}w\tt{1}$ do not belong to $\Fact(\mathbf{y})$. 

    First observe that since $\tt{0}w\tt{0}\in\Fact(y)$, we have $\tt{20}w\tt{02}\in\Fact(\mathbf{y})$, and we can write $\tt{20}w\tt{02}=g(v)\tt{2}$ for some word $v\in\Fact(\mathbf{x})$ of the form $\tt{0}v'\tt{0}$.  (Note in particular that $w=g(v')\tt{2}$.)  By \Cref{Lemma:PalindromicPreimage}, we see that $v$ and $v'$ are palindromes.  Since $\tt{00}\not \in \Fact(\mathbf{x})$, the word $v'$ begins with $\tt{1}$ or $\tt{2}$.  
 
    Now suppose, towards a contradiction, that $\tt{1}w\tt{2}\in\Fact(\mathbf{y})$.  Then we have $\tt{21}w\tt{2}\in\Fact(\mathbf{y})$, which means that $g^{-1}(\tt{21}w)=\tt{1}v'\in \Fact(\mathbf{x})$.  Since neither $\tt{11}$ nor $\tt{12}$ belongs to $\Fact(\mathbf{x})$, we see that $v'$ begins with $\tt{0}$.  But this contradicts the fact that $v'$ begins with $\tt{1}$ or $\tt{2}$.  So we conclude that $\tt{1}w\tt{2}\not\in\Fact(\mathbf{y})$.

    Finally, suppose that $\tt{2}w\tt{1}\in\Fact(\mathbf{y})$.  Then we have $g^{-1}(w\tt{1})=v'\tt{1}\in\Fact(\mathbf{y})$.  Since neither $\tt{11}$ nor $\tt{21}$ belongs to $\Fact(\mathbf{x})$, we see that $v'$ ends with $\tt{0}$.  But this contradicts the fact that $v'$ is a palindrome that begins with $\tt{1}$ or $\tt{2}$.  So we conclude that $\tt{2}w\tt{1}\not\in\Fact(\mathbf{y})$.
\end{proof}

\begin{proposition}\label{Prop:zPalindromicComplexity}
    $P_\mathbf{z}(n)=\begin{cases}
        1, &\text{ if $n=0$;}\\
        3, &\text{ if $n=1$ or $n=2$;}\\
        4, &\text{ if $n=3$, or $n\geq 4$ and $n$ is even;}\\
        2, &\text{ if $n\geq 5$ and $n$ is odd.}
    \end{cases}$
\end{proposition}

\begin{proof}
    By generating all factors of $\mathbf{z}$ of length at most $20$, we verify the formula given in the theorem statement for all $n\leq 20$.  Now it suffices to show that every palindromic factor of $\mathbf{z}$ of length at least $19$ has a unique palindromic extension in $\Fact(\mathbf{z})$.

    Let $w$ be a palindromic factor of $\mathbf{z}$ of length at least $19$, so that $w$ contains both $\tt{1}$ and $\tt{2}$.  We assume that $\tt{1}$ appears before $\tt{2}$ in $w$; a symmetric argument applies if $\tt{2}$ appears before $\tt{1}$.  Let $u\tt{00}$ be the longest prefix of $w$ that does not contain the letter $\tt{2}$.  Then $u$ must be a nonempty suffix of $\tau(\tt{0})$, $\tau(\tt{1})$, or $\tau(\tt{2})$, and we can write
    \begin{equation*}
        w=u\tt{00}w'\tt{00}\tilde{u},
    \end{equation*}
    where $w'$ begins and ends in $\tt{2}$ and $\tilde{u}$ is the reversal of $u$.  So we see that 
    \begin{equation*}
        w=u\overline{\tau}(y)\tt{00}\tilde{u}
    \end{equation*}
    for some word $y\in\Fact(\mathbf{y})$.  It is not hard to see that $y$ must be a palindrome of odd length.

    \smallskip

    \noindent
    \textbf{Case 1:} Suppose that $u=\tt{01}$.  First note that $\tt{2}w\tt{2}\not\in\Fact(\mathbf{z})$, since $\tt{201}\not\in\Fact(\mathbf{z})$.  Note further that $\tt{0}w\tt{0}\in\Fact(\mathbf{z})$ if and only if $\tt{0}y\tt{0}\in\Fact(\mathbf{y})$, and that $\tt{1}w\tt{1}\in\Fact(\mathbf{z})$ if and only if $ayb\in\Fact(\mathbf{y})$ for some $a,b\in\{\tt{1},\tt{2}\}$.

    Suppose first that $\tt{0}y\tt{0}\in\Fact(\mathbf{y})$.  Then we see from \Cref{Lemma:0w0} that the words $\tt{1}y\tt{1}$, $\tt{1}y\tt{2}$, $\tt{2}y\tt{1}$, and $\tt{2}y\tt{2}$ do not belong to $\Fact(\mathbf{y})$.  Hence, from the biconditional statements above, we conclude that $\tt{0}w\tt{0}\in\Fact(\mathbf{z})$ and $\tt{1}w\tt{1}\not\in\Fact(\mathbf{z})$, i.e., that $\tt{0}w\tt{0}$ is the unique palindromic extension of $w$ in $\Fact(\mathbf{z})$.

    Suppose otherwise that $\tt{0}y\tt{0}\not\in\Fact(\mathbf{y})$.  Then either $\tt{1}y\tt{1}$ or $\tt{2}y\tt{2}$ belongs to $\Fact(\mathbf{y})$, since $y$ has a unique palindromic extension in $\Fact(\mathbf{y})$ by \Cref{Prop:yPalindromicComplexity}.  Hence, from the biconditional statements above, we conclude that $\tt{1}w\tt{1}$ is the unique palindromic extension of $w$ in $\Fact(\mathbf{z})$.

    \smallskip

    \noindent
    \textbf{Case 2:} Suppose that $u=\tt{00101101}$.  First note that $\tt{0}w\tt{0}\not\in \Fact(\mathbf{z})$, since $\tt{000}\not\in\Fact(\mathbf{z})$. Note further that $\tt{1}w\tt{1}\in\Fact(\mathbf{z})$ if and only if $\tt{2}y\tt{2}\in\Fact(\mathbf{y})$, and that $\tt{2}w\tt{2}\in\Fact(\mathbf{z})$ if and only if $\tt{1}y\tt{1}\in\Fact(\mathbf{y})$.

    From the structure of $w$, we see that $\Fact(\mathbf{y})$ contains a factor of the form $ayb$, where $a,b\in\{\tt{1},\tt{2}\}$.  So by \Cref{Lemma:0w0}, we have $\tt{0}y\tt{0}\not\in\Fact(\mathbf{y})$.  But $y$ has a unique palindromic extension in $\Fact(\mathbf{y})$ by \Cref{Prop:yPalindromicComplexity}, so exactly one of the words $\tt{1}y\tt{1}$ and $\tt{2}y\tt{2}$ belongs to $\Fact(\mathbf{y})$.  By the biconditionals above, we conclude that $w$ has a unique palindromic extension in $\Fact(\mathbf{z})$.

    \smallskip
    
    \noindent
    \textbf{Case 3:} Say $u\not\in\{\tt{01},\tt{00101101}\}$.  It is not hard to show by inspection of $\tau(\tt{0})$, $\tau(\tt{1})$, and $\tau(\tt{2})$ that $w$ has a unique palindromic extension.  For example, if $u=\tt{1}$, then the unique palindromic extension of $w$ in $\Fact(\mathbf{z})$ is $\tt{0}w\tt{0}$, since $\tau(\tt{0})$, $\tau(\tt{1})$, and $\tau(\tt{2})$ have common suffix $\tt{01}$.  If $u=\tt{001}$, then $w=\tau(\tt{0}y\tt{0})\tt{00}$, and the unique palindromic extension of $w$ in $\Fact(\mathbf{z})$ is $\tt{2}w\tt{2}$.
\end{proof}

We now have the means to prove \Cref{Prop:Rich}.

\begin{proof}[Proof of \Cref{Prop:Rich}]
    First of all, when $n \leq 3$, it is straightforward to verify that the equation \eqref{Eq:RichCharacterization} holds. Let $n \geq 4$. Since each right-special factor of $\mathbf{z}$ of length $n$ has exactly two right extensions by \Cref{Prop:zFactorComplexity}, the quantity $C_\mathbf{z}(n+1) - C_\mathbf{z}(n)$ equals the number of right-special factors of $\mathbf{z}$ of length $n$. Hence the right side of \eqref{Eq:RichCharacterization} equals $6$. \Cref{Prop:zPalindromicComplexity} implies that the left side of the equation also equals $6$. Therefore $\mathbf{z}$ is rich by \Cref{Prop:RichCharacterization}.
\end{proof}

\section{The Critical Exponent}
\label{Section:CriticalExponent}

Throughout this section, let $\mathbf{x}=f^\omega(\tt{0})$, $\mathbf{y}=g(\mathbf{x})$, and $\mathbf{z}=\tau(\mathbf{y})$.  We will write $\mathbf{x}=x_0x_1x_2\cdots$, $\mathbf{y}=y_0y_1y_2\cdots$, and $\mathbf{z}=z_0z_1z_2\cdots$, where the $x_i$'s, $y_i$'s, and $z_i$'s are letters.

\begin{theorem}\label{Theorem:CriticalExponent}
    The critical exponent of $\mathbf{z}$ is 
    \begin{equation*}
    1+\frac{1}{3-\mu_1}\approx 2.25876324,
    \end{equation*}
    where $\mu_1\approx 2.20557$ is the unique real root of the polynomial $x^3-2x^2-1$.
\end{theorem}

We prove \Cref{Theorem:CriticalExponent} by adapting the method of Krieger~\cite{Krieger2007} for finding the critical exponent of a fixed point of a non-erasing morphism.  The basic idea is that every factor of exponent greater than $9/4$ in $\mathbf{z}$ belongs to one of only finitely many sequences of ``unstretchable'' repetitions, which cannot be extended periodically (i.e., ``stretched'') in $\mathbf{z}$.
Each of these sequences is obtained as follows:
\begin{itemize}
    \item We start from a short unstretchable repetition $w_0$ in $\mathbf{x}$ and repeatedly apply $f$ and stretch to the left and right as far as possible.  This gives us a sequence $w_0,w_1,w_2,\ldots$ of unstretchable repetitions in $\mathbf{x}$, which we call an ``inner stretch sequence''.
    \item We apply $g$ and $\tau$ to each word in the sequence $w_0,w_1,w_2\ldots$, and stretch to the left and right as far as possible.  This gives us a sequence $W_0,W_1,W_2,\ldots$ of unstretchable repetitions in $\mathbf{z}$, which we call an ``outer stretch sequence''.
\end{itemize}
Once we identify the outer stretch sequences that contain repetitions of exponent greater than $9/4$ in $\mathbf{z}$, we analyze the exponents of the words in these outer stretch sequences, and show that their supremum is $1 + 1/(3 - \mu_1)$.

In \Cref{Subsection:StretchSequences}, we formally define unstretchable repetitions and stretch sequences.  In \Cref{Subsection:9/4}, we show that every unstretchable repetition of exponent greater than $9/4$ in $\mathbf{z}$ belongs to one of only finitely many outer stretch sequences.  Finally, in \Cref{Subsection:CriticalExponent}, we show that the supremum of the exponents of the words in these outer stretch sequences is $1 + 1/(3-\mu_1)$.

\subsection{Unstretchable Repetitions and Stretch Sequences}\label{Subsection:StretchSequences}

In order to define an unstretchable repetition, we need to consider a specific occurrence of the repetition.  An \emph{occurrence} of a factor $w$ in an infinite word $\mathbf{u}=u_0u_1u_2\cdots$ is a triple $(w,i,j)$ such that $w=u_i\cdots u_{j}$.  We usually omit the reference to the triple, and simply refer to an occurrence $(w,i,j)$ as $w=u_i\cdots u_j$.  The set of all occurrences of all factors of $\mathbf{u}$ is denoted by $\Occ(\mathbf{u})$.  Note that every occurrence of a word $w$ in $\mathbf{x}$ corresponds in a natural way to an occurrence of $f(w)$ in $\mathbf{x}$.  More precisely, if $w=x_i\cdots x_{j}$, then $f(w)=x_k\cdots x_\ell$, where $k=|f(x_0\cdots x_{i-1})|$ and $\ell=|f(x_0\cdots x_j)|-1$.  Similarly, every occurrence of a word $w$ in $\mathbf{x}$ corresponds to an occurrence of $g(w)$ in $\mathbf{y}$, and every occurrence of a word $w$ in $\mathbf{y}$ corresponds to an occurrence of either $\tau(w)$ or $\overline{\tau}(w)$ in $\mathbf{z}$.  If $w=y_i\cdots y_j\in\Occ(\mathbf{y})$, then we write 
\begin{align*}
    \tilde{\tau}(w)=z_k\cdots z_{\ell}, 
\end{align*}
where $k=|\tau(y_0\cdots y_{i-1})|$ and $\ell=|\tau(y_0\cdots y_j)|-1$, i.e.,
\begin{align*}
    \tilde{\tau}(w)=\begin{cases}
        \tau(w), &\text{ if $i$ is even;}\\
        \overline{\tau}(w), &\text{ if $i$ is odd.}
    \end{cases}
\end{align*}

Suppose that $w=u_i\cdots u_j\in\Occ(\mathbf{u})$ has period $q$.  We say that the pair $(w,q)$ is \emph{left-stretchable} if $u_{i-1}\cdots u_j$ also has period $q$, i.e., if $u_{i-1}=u_{i+q-1}$.  Similarly, we say that $(w,q)$ is \emph{right-stretchable} if $u_i\cdots u_{j+1}$ has period $q$.  The \emph{left stretch} of $(w,q)$ is the longest word $\lambda$ such that $\lambda w$ has period $q$ and $\lambda w=u_{i-|\lambda|}\cdots u_{j}$.  The \emph{right stretch} of $(w,q)$ is the longest word $\rho$ such that $w\rho$ has period $q$ and $w\rho=u_i\cdots u_{j+|\rho|}$.  (Such a word $\rho$ exists as long as $\mathbf{u}$ has finite critical exponent.)   We say that $(w,q)$ is \emph{unstretchable} if it is neither left-stretchable nor right-stretchable, i.e., if $\lambda=\rho=\varepsilon$.  Sometimes we say that ``$w$ is unstretchable with respect to $q$'' instead of ``$(w,q)$ is unstretchable''.

Suppose that $w=u^r=x_i\cdots x_j\in\Occ(\mathbf{x})$ is unstretchable with respect to $q=|u|$.  Note that the word $f(w)$ has period $|f(u)|$, the word $f^2(w)$ has period $|f^2(u)|$, and so on.  Further, each time we apply $f$, we can stretch the resulting repetition as far as possible to the left and the right in $\mathbf{x}$, to obtain a sequence of unstretchable repetitions $w_0,w_1,w_2,\ldots$ with periods $|u|,|f(u)|,|f^2(u)|,\ldots,$ respectively.  Formally, the \emph{inner stretch sequence} of $(w,q)$ is the sequence 
\begin{equation*}
(w_0,q_0),(w_1,q_1),(w_2,q_2),\ldots
\end{equation*}
where $q_n=|f^n(u)|$ for all $n\geq 0$, and $w_n$ is defined recursively by $(w_0,q_0)=(w,q)$ and 
\begin{equation*}
w_{n}=\lambda_n f(w_{n-1})\rho_n
\end{equation*}
for all $n\geq 1$, where $\lambda_n$ (resp.\ $\rho_n$) is the left (resp.\ right) stretch of $(f(w_{n-1}),q_{n})$ in $\mathbf{x}$.  Thus we have
\begin{align*}
    w_0&=w,\\
    w_1&=\lambda_1f(w)\rho_1,\\
    w_2&=\lambda_2f(\lambda_1f(w)\rho_1)\rho_2,\\
    w_3&=\lambda_3f(\lambda_2f(\lambda_1f(w)\rho_1)\rho_2)\rho_3,
\end{align*}
and so on.  Since $w_n$ has period $q_n$, we can write $w_n=v_n^{r_n}$, where $v_n$ is the prefix of $w_n$ of length $q_n$ and
\begin{equation*}
r_n=\frac{|w_n|}{q_n}=\frac{|f^n(w)|+\sum_{k=1}^n|f^{n-k}(\lambda_k\rho_k)|}{|f^n(u)|}
\end{equation*}
is an exponent of $w_n$.  We call $r_0,r_1,r_2,\ldots$ the \emph{inner power sequence} of $(w,q)$.

Now suppose that $u$ has an even number of $\tt{2}$'s.  Then it is not hard to show that for all $n\geq 0$, the word $v_n$ has an even number of $\tt{2}$'s.  It follows that $g(w_n)$ has even period $|g(v_n)|=|g(f^n(u))|$ and in turn that $\tau(g(w_n))$ has period $|\tau(g(v_n))|=|\tau(g(f^n(u)))|$.  For all $n\geq 0$, define $Q_n=|\tau(g(f^n(u)))|$ and 
\begin{equation*}
    W_n=\lambda_n^* \tilde{\tau}(g(w_n))\rho_n^*
\end{equation*}
where $\lambda_n^*$ (resp.\ $\rho_n^*$) is the left (resp.\ right) stretch of $(\tilde{\tau}(g(w_n)),Q_n)$ in $\mathbf{z}$.  The sequence
\begin{equation*}
(W_0,Q_0), (W_1,Q_1), (W_2,Q_2),\ldots
\end{equation*}
is called the \emph{outer stretch sequence} of $(w,q)$.  
Since $W_n$ has period $Q_n$, we can write $W_n=V_n^{R_n}$, where $V_n$ is the prefix of $W_n$ of length $Q_n$ and 
\begin{equation*}
R_n=\frac{|W_n|}{Q_n}=\frac{|\tau(g(f^n(w)))|+\sum_{k=1}^n|\tau(g(f^{n-k}(\lambda_k\rho_k)))|+|\lambda_n^*\rho_n^*|}{|\tau(g(f^n(u)))|}
\end{equation*}
is an exponent of $W_n$.  We call $R_0,R_1,R_2,\ldots$ the \emph{outer power sequence} of $(w,q)$.


\begin{example}\label{Example:StretchSequence}
Consider the occurrence $w=x_2=\tt{0}$ in $\mathbf{x}$.  Note that $(w,1)$ is unstretchable, since $x_1,x_3\neq \tt{0}$.  Let us find the first few terms of the inner stretch sequence of $(w,1)$.  By definition, we have $w_0=w$ and $q_0=1$.  Next, we have $q_1=|f(\tt{0})|=2$ and $f(w_0)=x_5x_6=\tt{01}$.  Since $x_4=\tt{2}$ and $x_7x_8=\tt{02}$, we find $\lambda_1=\varepsilon$ and $\rho_1=x_7=\tt{0}$, so that 
\begin{equation*}
w_1=\lambda_1\cdot f(w_0)\cdot \rho_1=x_5x_6x_7=\tt{010}.
\end{equation*}
Continuing in this manner, we find $q_2=|f^2(\tt{0})|=5$ and  
\begin{align*}
    w_2&=\lambda_2 \cdot f(w_1)\cdot \rho_2=\tt{2}\cdot \tt{0102201}\cdot \tt{02}=(\tt{20102})^2,
\end{align*}
and $q_3=|f^3(\tt{0})|=11$ and
\begin{align*}
    w_3&=\lambda_3 \cdot f(w_2)\cdot \rho_3=\varepsilon \cdot (\tt{02010220102})^2\cdot \tt{0}=(\tt{02010220102})^{23/11}.
\end{align*}
Thus, the inner stretch sequence of $(w,1)$ begins with
\begin{align*}
(w_0,q_0)&=(\tt{0},1),\\
(w_1,q_1)&=\left((\tt{01})^{3/2},2\right),\\
(w_2,q_2)&=\left((\tt{20102})^2,5\right), \text{ and}\\
(w_3,q_3)&=\left((\tt{02010220102})^{23/11},11\right),
\end{align*}
and the inner power sequence of $(w,1)$ begins with 
\begin{equation*}
1,\tfrac{3}{2},2,\tfrac{23}{11}.
\end{equation*}
Now let us find the initial term of the outer stretch sequence of $(w,1)$.  We have $Q_0=|\tau(g(\tt{0}))|=19$ and 
\begin{align*}
W_0&=\lambda_0^*\cdot \tau(g(\tt{0}))\cdot \rho_0^*\\
&=\tt{02}\cdot \tau(\tt{20})\cdot \tau(\tt{2})\tt{0010}\\
&=(\tt{0200101101001011010})^{41/19}.
\end{align*}
Thus, the initial term in the outer stretch sequence of $(w,1)$ is
\begin{align*}
    (W_0,Q_0)&=\left((\tt{0200101101001011010})^{41/19},19\right).
\end{align*}
Continuing in this manner, one can show that the outer power sequence of $(w,1)$ begins with
\begin{equation*}
\tfrac{41}{19},\tfrac{94}{43}, \tfrac{210}{94}, \tfrac{467}{207},\ldots
\end{equation*}
In fact, we will see that this is the sequence that is responsible for the critical exponent of $\mathbf{z}$.  Note that $467/207 \approx 2.25604$ is just barely larger than $9/4$.
\end{example}

\subsection{Repetitions of Exponent Greater than 9/4 in \texorpdfstring{$\mathbf{z}$}{z}}\label{Subsection:9/4}

In this subsection, we show that every unstretchable repetition of exponent greater than $9/4$ in $\mathbf{z}$ belongs to one of only a narrow family of outer stretch sequences.  More precisely, we prove the following proposition.

\begin{proposition}\label{Proposition:UnstretchableStructure}
    Suppose that $W=U^R=z_k\cdots z_\ell\in\Occ(\mathbf{z})$ is unstretchable with respect to $Q=|U|$, where $R > 9/4$. Then $(W,Q)$ is in the outer stretch sequence of $(w,|w|)$ for some occurrence $w=x_i\cdots x_j\in\Occ(\mathbf{x})$, where the word $w$ belongs to the set
    \begin{align*}
        S&=\{\tt{0},\tt{1},\tt{22},\tt{202},\tt{1022},\tt{0220},\tt{2201},\tt{10202},\tt{02020},\tt{20201}\}.
    \end{align*}
\end{proposition}

First, we show that unstretchable repetitions of exponent greater than $9/4$ in $\mathbf{z}$ come from unstretchable repetitions of exponent greater than $2$ in $\mathbf{y}$ with periods of even length by applying $\tau$ (or $\overline{\tau}$) and stretching to the left and right.

\begin{lemma}\label{Lemma:UnstretchableInZ}
    Suppose that $W=U^R=z_k\cdots z_\ell\in \Occ(\mathbf{z})$ is unstretchable with respect to $Q=|U|$, where $R > 9/4$.  Then we have 
    \begin{equation*}
    W=\lambda\tau(w)\rho \ \ \ \text{ or }\ \ \  W=\lambda\overline{\tau}(w)\rho
    \end{equation*}
    for some words $\lambda,\rho\in\Occ(\mathbf{z})$ and a word $w=y_i\cdots y_j\in \Occ(\mathbf{y})$ such that
    \begin{itemize}
        \item $w=u^r$ has period $q=|u|$, where $u$ is a prefix of $w$ satisfying $|\tau(u)|=|U|$;
        \item $(w,q)$ is unstretchable;
        \item $u$ has even length; and
        \item $r>2$.
    \end{itemize}
\end{lemma}

\begin{proof}
    First of all, we check that $\mathbf{z}$ has no factor of the form $U^R$ where $R > 9/4$ and $|U|\leq 136$.  So we may assume that $|U|>136$.  Then $U$ contains both $\tt{1}$ and $\tt{2}$.  By \Cref{Lemma:SisterOccurs}, we may assume that $\tt{2}$ appears before $\tt{1}$ in $U$.  Let $\lambda\tt{00}$ be the longest prefix of $W$ that does not contain the letter $\tt{1}$, and let $\rho$ be the longest suffix of $W$ that does not contain both $\tt{1}$ and $\tt{2}$.  Notice that our assumptions imply that $\lambda$ exists and ends with $\tt{2}$.  Then $\lambda$ must be a suffix of $\overline{\tau}(a)$ for some $a\in\{\tt{0,1,2}\}$, and $\rho$ must be a prefix of $\tau(b)\tt{00}$ (or its sister) for some $b\in\{\tt{0,1,2}\}$, and we can write 
    \begin{equation*}
        W=\lambda \tau(w)\rho
    \end{equation*}  
    for some word $w=y_i\cdots y_j\in\Occ(\mathbf{y})$.  Since $W=U^R$ where $R>9/4$ and $|U|\geq 136$, we have 
    \begin{equation*}
    |W|>\tfrac{9}{4}|U|=2|U|+\tfrac{1}{4}|U|\geq 2|U|+34.
    \end{equation*}
    Further, since $W=\lambda\tau(w)\rho$ and $|\lambda|\leq 16$ and $|\rho|\leq 18$, we have
    \begin{equation*}
    |\tau(w)|=|W|-|\lambda|-|\rho|>2|U|.
    \end{equation*}
    Write $\tau(w)=U'V'$, where $U'$ is the prefix of $\tau(w)$ of length $|U|$.  Since $|\tau(w)|>2|U|$, we have $|V'|>|U'|$.
    Since $W$ has period $|U|$ and $\lambda$ ends in $\tt{2}$, we see that $U'$ ends in $\tt{2}$.  Similarly, since $U'$ starts in $\tt{001}$, we see that $V'$ starts in $\tt{001}$.  It follows that we can write $U'=\tau(u)$ and $V'=\tau(v)$ for some words $u,v\in\Fact(\mathbf{y})$.  Since $\tau$ is injective, we must have $w=uv$, and we see that $w$ has period $q=|u|$. Since $|V'|>|U'|$, we must have $|v|>|u|$, which means that $w=u^r$ for some rational number $r>2$. Since $U'=\tau(u)$ starts in $\tt{001}$ and ends in $\tt{2}$, we see that $u$ must have even length.

    It remains to show that $(w,q)$ is unstretchable.  Suppose that $(w,q)$ is left-stretchable.  (The proof is similar if $(w,q)$ is right-stretchable.)  Then $y_{i-1}\cdots y_j$ has period $q$, and it follows that $\tt{1}\overline{\tau}(y_{i-1}\cdots y_j)\rho$ has period $Q$ and properly contains $W=z_k\cdots z_\ell$, which contradicts the assumption that $(W,Q)$ is unstretchable.
\end{proof}

Next, we show that unstretchable repetitions in $\mathbf{y}$ with exponent greater than $2$ and period of even length come from unstretchable repetitions in $\mathbf{x}$ with exponent at least $2$ with an even number of $\tt{2}$'s in the repeated prefix.

\begin{lemma}\label{Lemma:UnstretchableInY}
    Suppose that $W=U^R=y_k\cdots y_\ell\in \Occ(\mathbf{y})$ is unstretchable with respect to $Q=|U|$, where $R>2$ and $U$ has even length.  Then we have 
    \begin{equation*}
    W=g(w)\tt{2}
    \end{equation*}
    where $w=x_i\cdots x_j\in\Occ(\mathbf{x})$ satisfies the following:
    \begin{itemize}
        \item $w=u^r$ has period $q=|u|$, where $u$ is a prefix of $w$ satisfying $|g(u)|=|U|$;
        \item $(w,q)$ is unstretchable;
        \item $u$ has an even number of $\tt{2}$'s; and
        \item $r\geq 2$.
    \end{itemize}
\end{lemma}

\begin{proof}
    First note that since $(W,Q)$ is unstretchable, it must begin and end in $\tt{2}$.  For if $y_k\in\{\tt{0,1}\}$, then $y_k=y_{k+Q}$, and since every occurrence of $\tt{0}$ or $\tt{1}$ in $\mathbf{y}$ is preceded by $\tt{2}$, we have $y_{k-1}=y_{k+Q-1}$, which contradicts the assumption that $W$ is unstretchable.  Similarly $y_\ell = \tt{2}$.  Since $R>2$, we may write $W=UV\tt{2}$, where $V$ is a nonempty prefix of $W$ satisfying $|V|\geq |U|$.  Since $W$ begins in $\tt{2}$, we see that $V$ begins in $\tt{2}$.  So we can write $U=g(u)$ and $V=g(v)$ for some words $u,v\in\Occ(\mathbf{x})$.  Since $U$ has even length, we see that $u$ must have an even number of $\tt{2}$'s.  Let $w=uv$, so that $W=UV\tt{2}=g(u)g(v)\tt{2}=g(w)\tt{2}$.  Now since $g(v)\tt{2}$ is a prefix of $g(w)$, we see that $v$ is a prefix of $w$.  Further, since $|V|\geq |U|$, we must have $|v|\geq |u|$.  Thus, we have $w=u^r$ for some rational number $r\geq 2$.

    It remains to show that $(w,q)$ is unstretchable, where $q=|u|$.  Write $w=x_i\cdots x_j$, and suppose towards a contradiction that $(w,q)$ is left-stretchable.  (The proof is similar if $(w,q)$ is right-stretchable.)  Then $x_{i-1}\cdots x_j$ has period $q$, and it follows that $g(x_{i-1}\cdots x_j)\tt{2}=g(x_{i-1})W$ has period $Q$, which contradicts the assumption that $(W,Q)$ is unstretchable.
\end{proof}

In order to prove \Cref{Proposition:UnstretchableStructure}, it remains to show that every unstretchable repetition in $\mathbf{x}$ with exponent at least $2$ and an even number of $\tt{2}$'s in the repeated prefix belongs to the inner stretch sequence of $(w,|w|)$ for some word $w\in S$.  We start by showing that every sufficiently long unstretchable repetition in $\mathbf{x}$ comes from a shorter unstretchable repetition in $\mathbf{x}$.

\begin{lemma}\label{Lemma:UnstretchablePreimage}
    Suppose that $W=x_k\cdots x_\ell\in\Occ(\mathbf{x})$ has period $Q$ and that $(W,Q)$ is unstretchable.  Write $W=UV$,  where $|U|=Q$, and suppose that $|U|\geq 3$ and $|V|\geq 3$.  Then $W=\lambda f(w)\rho$ for some words $\lambda,w,\rho\in \Occ(\mathbf{x})$ such that
    \begin{itemize}
        \item $\lambda\in\{\varepsilon,\tt{2}\}$ and $\rho\in\{\tt{0},\tt{02}\}$;
        \item $w=x_i\cdots x_j$ has period $q=|u|$, where $u$ is a prefix of $w$ satisfying $|f(u)|=|U|$; and
        \item $(w,q)$ is unstretchable. 
    \end{itemize}
\end{lemma}

\begin{proof}
    We begin by proving a simple claim.

    \begin{claim} 
    $W$ has prefix $\tt{0}$ or $\tt{201}$ and suffix $\tt{0}$ or $\tt{102}$.
    \end{claim}

    \begin{subproof} We prove only that $W$ has prefix $\tt{0}$ or $\tt{201}$.  The proof that $W$ has suffix $\tt{0}$ or $\tt{102}$ uses symmetric arguments.
    
    We begin by showing that $\tt{1}$ is not a prefix of $W$.  Suppose towards a contradiction that $W$ has prefix $\tt{1}$.  Then both $U$ and $V$ begin in $\tt{1}$.  But the unique left extension of $\tt{1}$ in $\Fact(\mathbf{x})$ is $\tt{01}$, which means that $x_{k-1}=x_{k+q-1}=\tt{0}$.  This contradicts the assumption that $W$ is unstretchable.  So $W$ has prefix $\tt{0}$ or $\tt{2}$.

    
    If $W$ has prefix $\tt{0}$, then we are done, so suppose that $W$ has prefix $\tt{2}$.  Then $W$ has prefix $\tt{201}$, $\tt{202}$, or $\tt{220}$, and $V$ has the same prefix of length $3$ as $W$, since $|V|\geq 3$.  The unique left extension of $\tt{202}$ in $\Fact(\mathbf{x})$ is $\tt{0202}$, and the unique left extension of $\tt{220}$ in $\Fact(\mathbf{x})$ is $\tt{0220}$.  So if $W$ has prefix $\tt{202}$ or $\tt{220}$, then $x_{k-1}=x_{k+q-1}=\tt{0}$, which contradicts the assumption that $W$ is unstretchable.  Thus, we conclude that $W$ must have prefix $\tt{201}$.
    \end{subproof}

    We now consider two cases, based on the prefix of $W$.

    \smallskip
    
    \noindent 
    \textbf{Case 1:} Suppose that $W$ has prefix $\tt{0}$.  
    Then both $U$ and $V$ have prefix $\tt{0}$.  Further, since $W$ has suffix $\tt{0}$ or $\tt{102}$, we can write $V=V'\rho$ where $\rho\in\{\tt{0},\tt{02}\}$.  Let $w=x_i\cdots x_j$ be the unique preimage of $UV'$ in $\mathbf{x}$.  Write $w=uv$, where $f(u)=U$ and $f(v)=V'$.  Let $\lambda=\varepsilon$, so that $W=\lambda f(w)\rho$.  Since $f(w)\rho=f(u)f(v)\rho$ has period $|U|=|f(u)|$, we see that $w=uv$ must have period $|u|=q$.
    
    It remains to show that $(w,q)$ is unstretchable.  Suppose first that $(w,q)$ is right stretchable.  But then $x_i\cdots x_{j+1}$ has period $q$, and it follows that $f(x_i\cdots x_{j+1})\tt{0}=x_k\cdots x_{\ell'}$ has period $Q$ and contains the occurrence $W=x_k\cdots x_\ell$ as a proper prefix.  But this contradicts the assumption that $(W,Q)$ is unstretchable.  Suppose now that $w$ is left stretchable, i.e., that $x_{i-1}\cdots x_{j}$ has period $q$.  But then $f(x_{i-1}\cdots x_j)\rho=f(x_{i-1})W=x_{k'}\cdots x_\ell$ has period $Q$, and contains the occurrence $W=x_k\cdots x_\ell$ as a proper suffix.  But this contradicts the assumption that $(W,Q)$ is unstretchable.

    \smallskip

    \noindent
    \textbf{Case 2:} Suppose that $W$ has prefix $\tt{201}$.  Then both $U$ and $V$ have prefix $\tt{201}$, and since $W$ has suffix $\tt{0}$ or $\tt{102}$, we can write $U=\tt{2}U'$ and $V=\tt{2}V'\rho$, where $\rho\in\{\tt{0},\tt{02}\}$.  So $W=\tt{2}U'\tt{2}V'\rho$, and both $U'$ and $V'$ have prefix $\tt{01}$.  Let $w=x_i\cdots x_{j}$ be the preimage of $U'\tt{2}V'$ in $f$.  Write $w=uv$, where $f(u)=U'\tt{2}$ and $f(v)=V'$.  Let $\lambda=\tt{2}$ so that we have $W=\lambda f(w)\rho$.  Since $f(w)\rho=f(u)f(v)\rho$ has period $|U|=|U'\tt{2}|=|f(u)|$, we see that $w=uv$ must have period $|u|=q$.

    It remains to show that $(w,q)$ is unstretchable.  Suppose first that $(w,q)$ is right stretchable, i.e., that $x_i\cdots x_{j+1}$ has period $q$.  Then $\tt{2}f(x_i\cdots x_{j+1})\tt{0}=x_k\cdots x_{\ell'}$ has period $Q$ and contains the occurrence $W=x_k\cdots x_\ell$ as a proper prefix.  But this contradicts the assumption that $(W,Q)$ is unstretchable.  Suppose now that $w$ is left stretchable, i.e., that $x_{i-1}\cdots x_{j}$ has period $q$.  Then $f(x_{i-1}\cdots x_{j})\rho=x_{k'}\cdots x_\ell$ has period $Q$ and contains the occurrence $W=x_k\cdots x_\ell$ as a proper suffix.  But this contradicts the assumption that $(W,Q)$ is unstretchable.
\end{proof}

The next lemma describes the left and right stretch of an occurrence $w=x_i\cdots x_j\in\Occ(\mathbf{x})$ based on the context in which it appears, i.e., based on $x_{i-1}\cdots x_{j+1}$.

\begin{lemma}\label{Lemma:UnstretchablePrefixAndSuffix}
    Suppose that $w=x_i\cdots x_j\in\Occ(\mathbf{x})$ has period $q$ and that $(w,q)$ is unstretchable, where $i\geq 1$.  Let $Q=|f(x_i\cdots x_{i+q-1})|$ be the corresponding period of $f(w)=x_k\cdots x_{\ell}$, and let $\lambda$ (resp.~$\rho$) be the left (resp.~right) stretch of $(f(w),Q)$.  Then 
    \begin{equation*}
        \lambda=\begin{cases}
            \tt{2}, &\text{ if $x_{i-1},x_{i+q-1}\in\{\tt{1,2}\}$;}\\
            \tt{\varepsilon}, &\text{ otherwise;}
        \end{cases}
    \end{equation*}
    and 
    \begin{equation*}
        \rho=\begin{cases}
            \tt{02}, &\text{ if $x_{j+1},x_{j-q+1}\in\{\tt{1,2}\}$;}\\
            \tt{0}, &\text{ otherwise.}
        \end{cases}
    \end{equation*}
\end{lemma}

\begin{figure}
    \centering
    \begin{tikzpicture}[scale=0.8]
        \draw (-1.5,0) rectangle node{$x_{i-1}$} (0,1);
        \draw (0,0) rectangle node{$w$} (13.5,1);
        \draw (13.5,0) rectangle node{$x_{j+1}$} (15,1);
        \draw (-1.5,-1) rectangle node{$x_{i-1}$} (0,0);
        \draw (0,-1) rectangle node{\footnotesize $x_i$} (1.5,0);
        \draw (1.5,-1) rectangle node{$\cdots$} (4.5,0);
        \draw (4.5,-1) rectangle node{\footnotesize $x_{i+q-1}$} (6,0);
        \draw (6,-1) rectangle node{$\cdots$} (7.5,0);
        \draw (7.5,-1) rectangle node{\footnotesize $x_{j-q+1}$} (9,0);
        \draw (9,-1) rectangle node{$\cdots$} (12,0);
        \draw (12,-1) rectangle node{\footnotesize $x_j$} (13.5,0);
        \draw (13.5,-1) rectangle node{$x_{j+1}$} (15,0);
        \draw [decorate,decoration={brace,amplitude=5pt,mirror,raise=5pt}]
  (0,-1) -- (6,-1) node[midway,below,yshift=-10pt]{length $q$};
        \draw [decorate,decoration={brace,amplitude=5pt,mirror,raise=5pt}]
  (7.5,-1) -- (13.5,-1) node[midway,below,yshift=-10pt]{length $q$};
    \end{tikzpicture}
    \caption{An occurrence $w=x_i\cdots x_j$ such that $(w,q)$ is unstretchable.  (We have illustrated the case that $|w|>2q$.)}
    \label{Figure:UnstretchableInContext}
\end{figure}

\begin{proof}
    First note that since $(w,q)$ is unstretchable, we must have $x_{i-1}\neq x_{i+q-1}$ and $x_{j+1}\neq x_{j-q+1}$ (see \Cref{Figure:UnstretchableInContext}).  If $x_{i-1},x_{i+q-1}\in\{\tt{1,2}\}$, then $\{x_{i-1},x_{i+q-1}\}=\{\tt{1,2}\}$, and it follows that $\lambda$ is the longest common suffix of $f(\tt{1})$ and $f(\tt{2})$, which is $\tt{2}$.  Otherwise, we have $x_{i-1}=\tt{0}$ (and $x_{i+q-1}\in\{\tt{1,2}\}$) or $x_{i+q-1}=\tt{0}$ (and $x_{i-1}\in\{\tt{1},\tt{2}\}$).  In either case, the longest common suffix of $f(x_{i-1})$ and $f(x_{i+q-1})$ is $\varepsilon$, hence $\lambda=\varepsilon$.

    The argument for $\rho$ is similar.  If $x_{j+1},x_{j-q+1}\in\{\tt{1,2}\}$, then $\{x_{j+1},x_{j-q+1}\}=\{\tt{1,2}\}$, and it follows that $\rho=\tt{02}$.  Otherwise, we have $x_{j+1}=\tt{0}$ (and $x_{j-q+1}\neq \tt{0}$) or $x_{j-q+1}=\tt{0}$ (and $x_{j+1}\neq \tt{0}$), and it follows that $\rho=\tt{0}$.
\end{proof}

We are finally ready to prove that every unstretchable repetition in $\mathbf{x}$ with exponent at least $2$ and an even number of $\tt{2}$'s in the repeated prefix belongs to the inner stretch sequence of $(w,|w|)$ for some word $w\in S$.

\begin{lemma}\label{Lemma:UnstretchableInX}
    Suppose that $W=U^R=x_k\cdots x_\ell\in\Occ(\mathbf{x})$ is unstretchable with respect to $Q=|U|$, where $R\geq 2$ and $U$ contains an even number of $\tt{2}$'s.  Then $(W,Q)$ is in the inner stretch sequence of $(w,|w|)$ for some word $w=x_i\cdots x_j\in\Occ(\mathbf{x})$, where $w$ belongs to the set
    \begin{align*}
        S&=\{\tt{0},\tt{1},\tt{22},\tt{202},\tt{1022},\tt{0220},\tt{2201},\tt{10202},\tt{02020},\tt{20201}\}.
    \end{align*}
\end{lemma}

\begin{proof}
    Let $w=u^r=uv=x_i\cdots x_j\in\Occ(\mathbf{x})$ be a shortest word such that $w$ is unstretchable with respect to $q=|u|$ and $(W,Q)$ is in the inner stretch sequence of $(w,q)$.  Note that $(W,Q)$ is in the inner stretch sequence of itself, so there is such a pair $(w,q)$.  We claim that $|u|<3$ or $|v|<3$. 
    Suppose otherwise that $|u|\geq 3$ and $|v|\geq 3$.  Then by \Cref{Lemma:UnstretchablePreimage}, we can write $w=\lambda f(w')\rho$ for some words $\lambda,w',\rho\in\Occ(\mathbf{x})$ such that $\lambda\in\{\varepsilon,\tt{2}\}$, $\rho\in\{\tt{0},\tt{02}\}$, $w'=(u')^{r'}$ where $u'$ is the prefix of $w'$ satisfying $|f(u')|=|u|$ and $w'$ is unstretchable with respect to $q'=|u'|$.  But as $w=\lambda f(w')\rho$ has period $q$, we see that $\lambda$ is the left stretch of $(f(w'),q)$, and $\rho$ is the right stretch of $(f(w'),q)$, so that $(w,q)$, and $(W,Q)$ in turn, are in the inner stretch sequence of $(w',q')$.  Since $|w'|<|w|$, this contradicts the minimality of $w$.

    So we have $|u|<3$ or $|v|<3$.  We now claim that $|u|\leq 5$.  Suppose otherwise that $|u|>5$, whence $|v|\leq 2$.  Then the $n$th term of the inner power sequence of $(w,q)$ is given by
    \begin{align}
        r_n&=\frac{|f^n(w)|+\sum_{k=1}^n|f^{n-k}(\lambda_k\rho_k)|}{|f^n(u)|}\nonumber\\
        &=\frac{|f^n(u)|+|f^n(v)|+\sum_{k=1}^n|f^{n-k}(\lambda_k\rho_k)|}{|f^n(u)|}\nonumber\\
        &=1+\frac{|f^n(v)|}{|f^n(u)|}+\frac{\sum_{k=1}^n|f^{n-k}(\lambda_k\rho_k)|}{|f^n(u)|} \label{Eq:rn}
    \end{align}
    Since $|v|\leq 2$ and $\tt{11}\not\in\Fact(\mathbf{x})$, it follows from \Cref{Lemma:LengthComp} that 
    \begin{align}\label{Ineq:fnv}
        |f^n(v)|\leq |f^n(\tt{01})|.
    \end{align}  
    By inspecting all factors of $\mathbf{x}$ of length $6$ and using \Cref{Lemma:LengthComp}, we find that 
    \begin{align}\label{Ineq:fnu}
        |f^n(u)|\geq 2|f^n(\tt{01})|.
    \end{align}
    Since $\lambda_k\in\{\varepsilon,\tt{2}\}$ and $\rho_k\in\{\tt{0},\tt{02}\}$, we have 
    \begin{align*}
        |f^{n-k}(\lambda_k\rho_k)|\leq |f^{n-k}(\tt{202})|=|f^{n-k+1}(\tt{1})|
    \end{align*} 
    for all $k\in\{1,2,\ldots n\}$.  Hence we have
    \begin{align}\label{Ineq:fnlambdarho}
        \sum_{k=1}^n|f^{n-k}(\lambda_k\rho_k)|\leq \sum_{k=1}^n|f^{n-k+1}(\tt{1})|=\sum_{\ell=1}^n|f^\ell(\tt{1})|< |f^n(\tt{01})|,
    \end{align}
    where the last inequality can be proven by a straightforward induction.
    Substituting inequalities (\ref{Ineq:fnv}), (\ref{Ineq:fnu}), and (\ref{Ineq:fnlambdarho}) into (\ref{Eq:rn}), we find
    \begin{align*}
        r_n< 1+\frac{|f^n(\tt{01})|}{2|f^n(\tt{01})|}+\frac{|f^n(\tt{01})|}{2|f^n(\tt{01})|}=2.
    \end{align*}
    But $W=U^R$ is in the inner stretch sequence of $(w,q)$, and we assumed that $R\geq 2$, so this is a contradiction.

    Therefore, we have $|u|\leq 5$, which means that there are only finitely many possibilities for $u$.  Since $U$ has an even number of $\tt{2}$'s and $U$ is a conjugate of $f^n(u)$ for some $n\geq 0$, we see that $u$ must also have an even number of $\tt{2}$'s.  For each factor $u$ of $\mathbf{x}$ of length at most $5$ with an even number of $\tt{2}$'s, we find the longest possible periodic extension $w=\lambda u\rho$ of $u$ in $\Fact(\mathbf{x})$, so that $\lambda$ is the left stretch of $(u,|u|)$ and $\rho$ is the right stretch of $(u,|u|)$ for some particular occurrence of $u$ in $\mathbf{x}$, and $(w,q)$ is unstretchable.  By inspection, the only such pairs $(w,q)$ that are not contained in the inner stretch sequence of $(w',q')$ for some word $w'$ with $|w'|<|w|$ belong to $S$ and satisfy $w=u$ (i.e., $v=\varepsilon$).
\end{proof}

\Cref{Proposition:UnstretchableStructure} now follows easily from \Cref{Lemma:UnstretchableInZ}, \Cref{Lemma:UnstretchableInY}, and \Cref{Lemma:UnstretchableInX}.

\subsection{The Critical Exponent}\label{Subsection:CriticalExponent}

In this subsection, we complete the proof of \Cref{Theorem:CriticalExponent}.  Essentially all that remains is to analyze the narrow family of outer power sequences described in \Cref{Proposition:UnstretchableStructure}.

For $w\in S$, we let $R_0(w),R_1(w),R_2(w),\ldots$ be the outer power sequence of $(w,|w|)$ for some occurrence $w=x_i\cdots x_j$ with $i\geq 1$.  We will see shortly that $R_n(w)$ is independent of $i$.  Note that we ignore the case $i=0$, where all of the left stretches involved in the stretch sequence of $w$ will be empty, but taking $i\geq 1$ gives us ``enough space'' to stretch as far as we need to on the left.  Since $\mathbf{x}$ is recurrent, we see that each term in the outer power sequence of an occurrence of $w$ starting at $i=0$ will be no greater than the corresponding term in the outer power sequence of some later occurrence of $w$.  Thus, since we are interested only in the largest exponents in $\mathbf{z}$, it is safe to ignore the case $i=0$.

In the next lemma, we derive an expression for $R_n(w)$ in terms of the lengths of the words $\tau(g(f^n(w)))$ and $\tau(g(f^k(\tt{0})))$ for $k\leq n$.  An important consequence is that $R_n(w)\leq R_n(\tt{0})$ for all $w\in S$ and $n\geq 0$, which means that the sequence $(R_n(\tt{0}))_n$ is responsible for the critical exponent of $\mathbf{z}$.

\begin{lemma}\label{Lemma:OuterPowerSequenceExpression}
    For all $w\in S$ and $n\geq 0$, we have
    \begin{align}\label{Equation:R2n}
        R_{2n}(w)=1+\frac{\sum_{k=0}^n|\tau(g(f^{2k}(\tt{0})))|+3}{|\tau(g(f^{2n}(w)))|}
    \end{align}
    and
    \begin{align}\label{Equation:R2n+1}
        R_{2n+1}(w)=1+\frac{\sum_{k=0}^n |\tau(g(f^{2k+1}(\tt{0})))|+10}{|\tau(g(f^{2n+1}(w)))|}.
    \end{align}
    It follows that $R_n(w)\leq R_n(\tt{0})$ for all $w\in S$ and $n\geq 0$.
\end{lemma}

\begin{proof}
    Let $w=x_i\cdots x_j$ be an occurrence of $w$ in $\mathbf{x}$ with $i\geq 1$.  Throughout this proof, we use the notation of \Cref{Subsection:StretchSequences} to describe the inner and outer stretch sequences of $(w,|w|)$.  In particular, we let $(w_0,q_0),(w_1,q_1),\ldots$ and $(W_0,Q_0),(W_1,Q_1),\ldots$ denote the inner and outer stretch sequences of $(w,|w|)$, and we let $\lambda_k$, $\lambda_k^*$, $\rho_k$, and $\rho_k^*$ denote the left and right stretches applied along the way.  For all $k\geq 0$, we write 
    \begin{equation*}
        w_k=x_{i_k}\cdots x_{j_k}.
    \end{equation*}
    We start by proving the following claim.
    \begin{claim}\label{Claim:LeftAndRightStretches}
    \begin{enumerate}[label=\normalfont(\alph*)]
        \item For all $k\geq 1$, we have $\lambda_k=\begin{cases}
            \varepsilon, & \text{ if $k$ is odd;}\\
            \tt{2}, & \text{ if $k$ is even.}
        \end{cases}$
        \item For all $k\geq 1$, we have $\rho_k=\begin{cases}
            \tt{0}, & \text{ if $k$ is odd;}\\
            \tt{02}, & \text{ if $k$ is even.}
        \end{cases}$
        \item For all $k\geq 0$, we have $|\lambda_k^*\rho_k^*|=\begin{cases}
            34, & \text{ if $k$ is odd;}\\ 
            22, & \text{ if $k$ is even.} 
        \end{cases}$
    \end{enumerate}
    \end{claim}

\begin{subproof}
    For parts (a) and (b), we proceed by induction on $k$.  For the base case, we inspect each word $w$ in $S$ individually and examine all possible contexts $x_{i-1}\cdots x_{j+1}$ in which the occurrence $w=x_i\cdots x_j$ might appear.  In every case, we see by  \Cref{Lemma:UnstretchablePrefixAndSuffix} that $\lambda_1=\varepsilon$ and $\rho_1=\tt{0}$.
    
    Now suppose that (a) and (b) hold for some integer $k\geq 1$.  If $k$ is odd, then we have $\lambda_k=\varepsilon$ and $\rho_k=\tt{0}$.  Thus we can write
    \begin{equation*}
        w_k=x_{i_k}\cdots x_{j_k}=f(w_{k-1})\tt{0},
    \end{equation*}
    So we see that $x_{i_k}=x_{i_k+q_k}=\tt{0}$.  Since $\tt{00}\not\in\Fact(\mathbf{x})$ and $(w_k,q_k)$ is unstretchable, we must have $\{x_{i_k-1},x_{i_k+q_k-1}\}=\{\tt{1},\tt{2}\}$.  Therefore, by \Cref{Lemma:UnstretchablePrefixAndSuffix}, we have $\lambda_{k+1}=\tt{2}$.  Similarly, we have $x_{j_k}=x_{j_k-q_k}=\tt{0}$, hence $\{x_{j_k+1},x_{j_k-q_k+1}\}=\{\tt{1,2}\}$ and $\rho_{k+1}=\tt{02}$.  On the other hand, if $k$ is even, then we have $\lambda_k=\tt{2}$ and $\rho_k=\tt{02}$.  Thus we can write
    \begin{equation*}
        w_k=x_{i_k}\cdots x_{j_k}=\tt{2}f(w_{k-1})\tt{02}.
    \end{equation*}
    So we see that $x_{i_k}=x_{i_k+q_k}=\tt{2}$.  Since $\tt{12}\not\in\Fact(\mathbf{x})$ and $(w_k,q_k)$ is unstretchable, we must have $\{x_{i_k-1},x_{i_k+q_k-1}\}=\{\tt{0},\tt{2}\}$.  Therefore, by \Cref{Lemma:UnstretchablePrefixAndSuffix}, we have $\lambda_{k+1}=\varepsilon$. 
    Similarly, we have $x_{j_k}=x_{j_k-q_k}=\tt{2}$, hence $\{x_{j_k+1},x_{j_k-q_k+1}\}=\{\tt{0},\tt{2}\}$ and $\rho_{k+1}=\tt{0}$.

    Finally, we prove (c).  First suppose that $k$ is odd.  Using (a) and (b), we see, as in the second paragraph of the proof, that $x_{i_k}=x_{i_k+q_k}=\tt{0}$ and $\{x_{i_k-1},x_{i_k+q_k-1}\}=\{\tt{1},\tt{2}\}$. Since $w \in S$, we see that $|g(w_k)|$ is even. Therefore $\lambda_k^*$ is the longest common suffix of $\tau(g(\tt{1}))$ and $\overline{\tau}(g(\tt{2}))$ or their sisters (observe that $|g(\tt{1})|$ is even and $|g(\tt{2})|$ is odd). This suffix is $\tt{00202202}$ or its sister, i.e., $|\lambda_k^*| = 8$. Similarly, from (a) and (b), we see that $x_{j_k}=x_{j_k-q_k}=\tt{0}$, so that $\{x_{j_k+1},x_{j_k-q_k+1}\}=\{\tt{1},\tt{2}\}$, and it follows that $\rho_k^*$ is the longest common prefix of $\tau(g(\tt{1})\tt{2})$ and $\tau(g(\tt{2})\tt{2})$ (or their sisters), which has length $26$.  So we conclude that $|\lambda_k^*\rho_k^*|=8+26=34$.

    Now suppose that $k$ is even.  First suppose that $k\geq 2$.  From (a) and (b), we see that $x_{i_k}=x_{i_k+q_k}=\tt{2}$ and $x_{j_k}=x_{j_k-q_k}=\tt{2}$.  So $\{x_{i_k-1},x_{i_k+q_k-1}\}=\{\tt{0},\tt{2}\}$, and we see that $\lambda_k^*$ is the longest common suffix of $\tau(g(\tt{0}))$ and $\overline{\tau}(g(\tt{2}))$ (or their sisters), which has length $2$.  Similarly, we have $\{x_{j_k+1},x_{j_k-q_k+1}\}=\{\tt{0},\tt{2}\}$, and it follows that $\rho_k^*$ is the longest common prefix of $\tau(g(\tt{0})\tt{2})$ and $\tau(g(\tt{2})\tt{2})$ (or their sisters), which has length $20$.  Finally, in the case that $k=0$, we examine the context $x_{i-1}\cdots x_{j+1}$ in which $w$ appears, and confirm by inspection that $\lambda_0^*$ and $\rho_0^*$ are the same as $\lambda_k^*$ and $\rho_k^*$ for $k\geq 2$.  Thus, we conclude that $|\lambda_k^*\rho_k^*|=2+20=22$.  
\end{subproof}

    We now proceed with the proof of (\ref{Equation:R2n}).  Let $n\geq 0$. By definition, we have 
    \begin{align}\label{Equation:R2ndef}
        R_{2n}(w)&=\frac{|\tau(g(f^{2n}(w)))|+\sum_{k=1}^{2n}|\tau(g(f^{2n-k}(\lambda_k\rho_k)))|+|\lambda_{2n}^*\rho_{2n}^*|}{|\tau(g(f^{2n}(w)))|}
    \end{align}
    We first show that 
    \begin{align}\label{Equation:SumOfStretches}
        \sum_{k=1}^{2n}|\tau(g(f^{2n-k}(\lambda_k\rho_k)))|=\sum_{k=1}^n|\tau(g(f^{2k}(\tt{0})))|.
    \end{align}
    Starting from the left side, we group the terms in pairs, and then use \Cref{Claim:LeftAndRightStretches}(a) and (b), and the facts that $f(\tt{0})=\tt{01}$ and $f(\tt{1})=\tt{022}$, as follows.
    \begin{align*}
        \sum_{k=1}^{2n}|\tau(g(f^{2n-k}(\lambda_k\rho_k)))|&=\sum_{\ell=1}^n\left(|\tau(g(f^{2n-2\ell}(\lambda_{2\ell}\rho_{2\ell})))|+|\tau(g(f^{2n-2\ell+1}(\lambda_{2\ell-1}\rho_{2\ell-1})))| \right)\\
        &=\sum_{\ell=1}^n\left(|\tau(g(f^{2n-2\ell}(\tt{202})))|+|\tau(g(f^{2n-2\ell+1}(\tt{0})))| \right)\\
        &=\sum_{\ell=1}^n\left(|\tau(g(f^{2n-2\ell+1}(\tt{1})))|+|\tau(g(f^{2n-2\ell+1}(\tt{0})))| \right)\\
        &=\sum_{\ell=1}^n|\tau(g(f^{2n-2\ell+2}(\tt{0})))|\\
        &=\sum_{k=1}^n|\tau(g(f^{2k}(\tt{0})))|
    \end{align*}
    Substituting (\ref{Equation:SumOfStretches}) into (\ref{Equation:R2ndef}) and using \Cref{Claim:LeftAndRightStretches}(c) and the fact that $|\tau(g(\tt{0}))|=19$, we obtain
    \begin{align*}
        R_{2n}(w)&=1+\frac{\sum_{k=1}^n|\tau(g(f^{2k}(\tt{0})))|+22}{|\tau(g(f^{2n}(w)))|}=1+\frac{\sum_{k=0}^n|\tau(g(f^{2k}(\tt{0})))|+3}{|\tau(g(f^{2n}(w)))|}.
    \end{align*}
    The analogous expression (\ref{Equation:R2n+1}) for $R_{2n+1}$ can be established in a similar manner; we omit the details.

    Finally, by \Cref{Lemma:LengthComp}, we have $|\tau(g(f^n(w)))|\geq |\tau(g(f^n(\tt{0})))|$ for all $w\in S$.  Thus we see that $R_{2n}(w)\leq R_{2n}(\tt{0})$ for all $n\geq 0$. 
\end{proof}

With \Cref{Lemma:OuterPowerSequenceExpression} in hand, we now wish to find a simple expression for $|\tau(g(f^n(\tt{0})))|$.  We do so by finding a linear recurrence for $|\tau(g(f^n(\tt{0})))|$.

\begin{lemma}\label{Lemma:Recurrence}
    Write $a_n = |\tau(g(f^n(\tt{0})))|$ for $n \geq 0$. The sequence $(a_n)$ satisfies the recurrence
    \begin{equation}\label{Eq:Recurrence}
        a_n = 2a_{n-1} + a_{n-3}
    \end{equation}
    with initial values $19$, $43$, $94$.
\end{lemma}
\begin{proof}
    Recall that the Parikh vector $P(w)$ of a ternary word $w \in \{\tt{0}, \tt{1}, \tt{2}\}^*$ is the vector $(|w|_{\tt{0}}, |w|_{\tt{1}}, |w|_{\tt{2}})^T$. Let
    \begin{equation*}
        M_f =
        \begin{pmatrix}
            1 & 1 & 1 \\
            1 & 0 & 0 \\
            0 & 2 & 1
        \end{pmatrix}
        \quad \text{and} \quad
        M_g =
        \begin{pmatrix}
            1 & 0 & 0 \\
            0 & 1 & 0 \\
            1 & 1 & 1
        \end{pmatrix}
    \end{equation*}
    be the incidence matrices of the morphisms $f$ and $g$, so that $P(f(w)) = M_f P(w)$ and $P(g(w)) = M_g P(w)$ for all $w \in \{\tt{0}, \tt{1}, \tt{2}\}^*$. From the form of the transducer $\tau$, it is straightforward to see that $a_n = (3, 8, 16) P(g(f^n(0))) = (3, 8, 16) M_g M_f^n (1, 0, 0)^T$.
    
    The characteristic polynomial of the matrix $M_f$ equals $x^3 - 2x - 1$ so, by the Cayley-Hamilton Theorem, we have $M_f^3 - 2M_f^2 - I = 0$. Thus $0 = (3, 8, 16)M_g (M_f^3 - 2M_f^2 - I) M_f^n (1, 0, 0)^T = a_{n+3} - 2a_{n+2} - a_n$ for all $n \geq 0$. This proves the claim (after checking that the initial values are $19$, $43$, $94$).
\end{proof}

\Cref{Lemma:Recurrence} gives us a linear recurrence for $a_n=|\tau(g(f^n(\tt{0})))|$.  Using a computer algebra system to solve the recurrence, we obtain the following.

\begin{corollary}\label{Corollary:RecurrenceSolve}
    For all $n\geq 0$, we have
    \begin{equation*}
        |\tau(g(f^n(\tt{0})))|=\kappa_1\mu_1^n+\kappa_2\mu_2^n+\kappa_3\mu_3^n=\kappa_1\mu_1^n+2\Real\left(\kappa_2\mu_2^n\right),
    \end{equation*}
    where 
    \begin{align*}
        \mu_1&\approx 2.20557,\\
        \mu_2&\approx -0.10278+0.66546i, \text{ and}\\
        \mu_3&=\overline{\mu}_2
    \end{align*}
    are the roots of the polynomial $x^3-2x^2-1$, and
    \begin{align*}
        \kappa_1&\approx 19.31167,\\
        \kappa_2&\approx -0.15583-0.28157i, \text{ and}\\
        \kappa_3&=\overline{\kappa}_2.
    \end{align*}
\end{corollary}


We are now ready to prove that $\ce(\mathbf{z})=1+1 / (3-\mu_1)$.

\begin{proof}[Proof of \Cref{Theorem:CriticalExponent}]
    The critical exponent of $\mathbf{z}$ is the supremum of the exponents of all unstretchable repetitions in $\mathbf{z}$.  We have already seen in \Cref{Example:StretchSequence} that $\mathbf{z}$ has factors of exponent greater than $\tfrac{9}{4}$, and by \Cref{Proposition:UnstretchableStructure}, every unstretchable repetition with exponent greater than $\tfrac{9}{4}$ in $\mathbf{z}$ belongs to the outer stretch sequence of $(w,|w|)$ for some $w=x_i\cdots x_j\in\Occ(\mathbf{x})$ with $w\in S$.  It follows that
    \begin{equation*}
        \ce(\mathbf{z})=\sup\{R_n(w)\colon w\in S, n\geq 0\}.
    \end{equation*}
    By \Cref{Lemma:OuterPowerSequenceExpression}, we have $R_n(w)\leq R_n(\tt{0})$ for all $w\in S$ and $n\geq 0$, hence
    \begin{equation*}
        \ce(\mathbf{z})=\sup\{R_n(\tt{0})\colon n\geq 0\}.
    \end{equation*}
    So we need to show that $\sup\{R_n(\tt{0})\colon n\geq 0\}=1+1/(3-\mu_1)$.
    It suffices to show that 
    \begin{align}
    \lim_{n\rightarrow \infty}R_{2n}=1+\frac{1}{3-\mu_1}\label{Equation:R2nLimit}
    \end{align}
    and that
    \begin{align}
    R_n\leq 1+\frac{1}{3-\mu_1} \text{ for all $n\geq 0$.}\label{Equation:RnUpperBound}
    \end{align}
    
    We first establish (\ref{Equation:R2nLimit}).  By \Cref{Lemma:OuterPowerSequenceExpression}, we have 
    \begin{equation*}
        R_{2n}=1+\frac{\sum_{k=0}^n|\tau(g(f^{2k}(\tt{0})))|+3}{|\tau(g(f^{2n}(\tt{0})))|}.
    \end{equation*}
    Evidently, we have
    \begin{equation*}
        \lim_{n\rightarrow \infty}R_{2n}=1+\lim_{n\rightarrow \infty}S_{2n},
    \end{equation*}
    where
    \begin{equation*}
        S_{2n}=\frac{\sum_{k=0}^n|\tau(g(f^{2k}(\tt{0})))|}{|\tau(g(f^{2n}(\tt{0})))|},
    \end{equation*}
    as long as $\displaystyle\lim_{n\rightarrow \infty} S_{2n}$ exists.
    By \Cref{Corollary:RecurrenceSolve}, we have
    \begin{align*}
        \lim_{n\rightarrow \infty}S_{2n}&=
        \lim_{n\rightarrow \infty}\frac{\sum_{k=0}^n(\kappa_1\mu_1^{2k}+\kappa_2\mu_2^{2k}+\kappa_3\mu_3^{2k})}{\kappa_1\mu_1^{2n}+\kappa_2\mu_2^{2n}+\kappa_3\mu_3^{2n}}
    \end{align*}
    Breaking the sum in the numerator into three geometric sums, we obtain
    \begin{align*}
        \lim_{n\rightarrow \infty}S_{2n}&=\lim_{n\rightarrow \infty}\frac{\kappa_1\cdot \frac{\mu_1^{2n+2}-1}{\mu_1^2-1}+\kappa_2\cdot \frac{\mu_2^{2n+2}-1}{\mu_2^2-1}+\kappa_3\cdot \frac{\mu_3^{2n+2}-1}{\mu_3^2-1}}{\kappa_1\mu_1^{2n}+\kappa_2\mu_2^{2n}+\kappa_3\mu_3^{2n}}
    \end{align*}
    Finally, dividing through by $\kappa_1\mu_1^{2n}$ in the numerator and denominator, and using the fact that $|\mu_1|>1$ and $|\mu_2|,|\mu_3|<1$, we find
    \begin{align*}
        \lim_{n\rightarrow\infty}S_{2n}&=\lim_{n\rightarrow \infty}\frac{\frac{1}{\mu_1^2-1}\cdot \frac{\mu_1^{2n+2}-1}{\mu_1^{2n}}+\frac{\kappa_2}{\kappa_1}\cdot \frac{1}{\mu_2^2-1}\cdot \frac{\mu_2^{2n+2}-1}{\mu_1^{2n}}+\frac{\kappa_3}{\kappa_1}\cdot \frac{1}{\mu_3^2-1}\cdot \frac{\mu_3^{2n+2}-1}{\mu_1^{2n}}}{1+\frac{\kappa_2}{\kappa_1}\cdot\frac{\mu_2^{2n}}{\mu_1^{2n}}+\frac{\kappa_3}{\kappa_1}\cdot\frac{\mu_3^{2n}}{\mu_1^{2n}}}\\
        &=\frac{\mu_1^2}{\mu_1^2-1} = \frac{1}{3 - \mu_1}.
    \end{align*}

    Now we establish (\ref{Equation:RnUpperBound}).  Let $n\geq 0$.  We show only that 
    \begin{align}
        R_{2n}\leq 1+\frac{\mu_1^2}{\mu_1^2-1}.  \label{Equation:R2nA}
    \end{align}
    The proof for $R_{2n+1}$ is similar.  By \Cref{Lemma:OuterPowerSequenceExpression}, we see that (\ref{Equation:R2nA}) is equivalent to
    \begin{align}
        \frac{3+\sum_{k=0}^n|\tau(g(f^{2k}(\tt{0})))|}{|\tau(g(f^{2n}(\tt{0})))|}\leq \frac{\mu_1^2}{\mu_1^2-1},\label{Equation:R2nB}
    \end{align}
    and by \Cref{Corollary:RecurrenceSolve}, this is equivalent to
    \begin{align}\label{Equation:R2nC}
    \frac{3+\kappa_1\sum_{k=0}^n\mu_1^{2k}+2\Real\left(\kappa_2\sum_{k=0}^n\mu_2^{2k}\right)}{\kappa_1\mu_1^{2n}+2\Real\left(\kappa_2\mu_2^{2n}\right)}\leq \frac{\mu_1^2}{\mu_1^2-1}.
    \end{align}
    By evaluating the geometric sums and doing some basic algebra, we find that (\ref{Equation:R2nC}) is equivalent to 
    \begin{align}\label{Equation:R2nD}
        3+2\Real\left(\kappa_2\cdot\frac{\mu_2^{2n+2}-1}{\mu_2^2-1}\right)\leq \frac{1}{\mu_1^2-1}\left(\kappa_1+2\mu_1^2\Real\left(\kappa_2\mu_2^{2n}\right)\right).
    \end{align}
    So it suffices to prove (\ref{Equation:R2nD}).  Starting from the left side and using the fact that $|\mu_2|<1$ and $n\geq 0$, we have
    \begin{align*}
        3+2\Real\left(\kappa_2\cdot\frac{\mu_2^{2n+2}-1}{\mu_2^2-1}\right)&\leq 3+2\cdot |\kappa_2|\cdot \frac{|\mu_2|^{2n+2}+1}{|\mu_2^2-1|}\\
        &\leq 3+2\cdot|\kappa_2|\cdot \frac{|\mu_2|^2+1}{|\mu_2^2-1|}. 
    \end{align*}
    Starting from the right side, we have
    \begin{align*}
        \frac{1}{\mu_1^2-1}\left(\kappa_1+2\mu_1^2\Real\left(\kappa_2\mu_2^{2n}\right)\right)
        &\geq \frac{1}{|\mu_1^2-1|}\left(|\kappa_1|-2\cdot|\mu_1|^2\cdot |\kappa_2|\cdot |\mu_2|^{2n}\right)\\
        &\geq \frac{1}{|\mu_1^2-1|}\left(|\kappa_1|-2\cdot|\mu_1|^2\cdot |\kappa_2|\right).
    \end{align*}
    Finally, we use the values of $\kappa_1$, $\kappa_2$, $\mu_1$, and $\mu_2$ obtained in proving \Cref{Corollary:RecurrenceSolve} to verify computationally that
    \begin{align*}
        3+2\cdot|\kappa_2|\cdot \frac{|\mu_2|^2+1}{|\mu_2^2-1|}\leq \frac{1}{|\mu_1^2-1|}\left(|\kappa_1|-2\cdot|\mu_1|^2\cdot |\kappa_2|\right).
    \end{align*}
    So we conclude that (\ref{Equation:R2nD}) holds, as desired.
\end{proof}

\section{Alphabets of Size 4 and More}\label{Section:Conclusion}
In this section, we briefly discuss the repetition threshold for rich words over alphabets of size $4$ and more. By \cite[Thm.~2]{CurrieMolRampersad2020} and \Cref{Theorem:CriticalExponent}, we have $\RT(R_2) = 2 + \sqrt{2}/2 \approx 2.7071$ and $\RT(R_3) = 1 + 1 / (3 - \mu_1) \approx 2.2588$. Moreover, $\RT^*(R_k) = 2$ for all $k \geq 2$ \cite{DvorakovaKloudaPelantova2024}. It is natural to ask the following question.

\begin{question}
    What is the number $\RT(R_4)$?
\end{question}

By performing an exhaustive computational search, we have shown that $\RT(R_4) > 2.117$ (a longest $2.117$-power-free rich word over $4$ letters has length $46628$). Since we have found a $2.12$-power-free rich word of length one million over $4$ letters, we believe this to be close to the true value of $\RT(R_4)$, but we have not attempted to work on this.\footnote{The code used to verify these claims, and a text file containing a $2.12$-power-free rich word of length one million over $4$ letters, can be found at \url{https://github.com/japeltom/rich-repetition-threshold}.}

Based on the known results and the computational evidence we have gathered, we propose the following conjecture.

\begin{conjecture}
    $\displaystyle\lim_{k\to\infty} \RT(R_k) = 2$.
\end{conjecture}

\subsection*{Acknowledgements}

We thank the anonymous referees for suggestions which helped to improve the paper.  Research of James D.~Currie is supported by NSERC grant DDG-2024-00005.  Research of Lucas Mol is supported by NSERC grant RGPIN-2021-04084.


\end{document}